\documentclass[reqno,11pt,a4paper]{amsart}
\usepackage[T1]{fontenc}
\usepackage{graphicx,mathtools,amssymb,amsthm,thmtools}
\usepackage{xcolor,babel,amsmath,enumitem,bm}
\usepackage{appendix}
\usepackage{tikz-cd,caption}
\usetikzlibrary{patterns, arrows.meta, positioning}
\usepackage[numbers]{natbib}
\newtheorem{theorem}{Theorem}[section]
\newtheorem{lemma}{Lemma}[section]

\newtheorem{proposition}{Proposition}[section]
\newtheorem{remark}{Remark}[section]
\numberwithin{equation}{section}
\usepackage[%
colorlinks=true,  % 启用彩色链接而非边框
linkcolor=blue,  % 内部链接颜色
citecolor=blue,  % 引用链接颜色
urlcolor=black,   % URL链接颜色
anchorcolor=black,% 锚点链接颜色
]{hyperref}
\makeatletter
\@ifundefined{theHtheorem}{}{%
}
\@ifundefined{theHlemma}{}{%
}
\@ifundefined{theHremark}{}{%
}
\makeatother
\usepackage{mathrsfs}
\allowdisplaybreaks[4]
\begin{document}
	\title[Transport Equation Theory in the Triebel-Lizorkin Spaces]
	{Transport equation theory in the Triebel-Lizorkin spaces and its applications to the ideal fluid flows }
	\author{Qianyuan Zhang}
	\address{Qianyuan Zhang\newline
		School of Mathematics and Statistics\\
		Huazhong University of Science and Technology, Wuhan 430074,  China}
	\email{qianyuanzhang@hust.edu.cn}
	
	\author{Kai Yan}
	\address{Kai Yan (Corresponding author)\newline
		School of Mathematics and Statistics\\
		Huazhong University of Science and Technology, Wuhan 430074,  China}
	\email{kaiyan@hust.edu.cn}
	
	\begin{abstract}
		In this paper, we develop a general theory for the transport equation within the framework of Triebel-Lizorkin spaces. We first derive commutator estimates in these spaces, dispensing with the conventional divergence-free condition, via the Bony paraproduct decomposition and vector-valued maximal function inequalities. Building on these estimates and combining the method of characteristics with a compactness argument, we then obtain the new a priori estimates and prove local well-posedness for the transport equation in Triebel-Lizorkin spaces. The resulting theory is applicable to a wide range of evolution equations, including models for incompressible and compressible ideal fluid flows, shallow water waves, among others. As an illustration, we consider the incompressible ideal magnetohydrodynamics (MHD) system. Employing the general transport theory developed here yields a complete local well-posedness result in the sense of Hadamard, covering both sub-critical and critical regularity regimes, and provides corresponding blow-up criteria for the ideal MHD equations in Triebel-Lizorkin spaces. Our results refine and substantially extend earlier work in this direction.
	\end{abstract}

	\maketitle
	
	
	\vskip 0.2cm
	
	
	\setcounter{equation}{0}

	\section{Introduction}
	\subsection{General theory for transport equation.}
	The primary objective of this paper is to develop a well-posedness theory for the following transport equation:
	\begin{equation}\label{Eq}
		\begin{cases}
		    \partial_tf+v\cdot\nabla f=g,& t > 0, x \in \mathbb{R}^d,\\
			f|_{t=0}=f_0,
		\end{cases}\tag{T}
	\end{equation}
	in the framework of Triebel-Lizorkin spaces, where $v\colon\mathbb{R}^+ \times \mathbb{R}^d \to \mathbb{R}^d$ is a given time-dependent vector field, $f_0\colon \mathbb{R}^d \to \mathbb{R}^N$ and $g\colon \mathbb{R}^+ \times \mathbb{R}^d \to \mathbb{R}^N$ $(d,N\in\mathbb{N}^+)$ are the given initial data and source term, respectively. Transport equations of this type arise naturally in many areas of mathematical physics, particularly in the analysis of partial differential equations from fluid mechanics. Although in such contexts the velocity field $v$ and the source $g$ often depend nonlinearly on the unknown $f$, having a good theory for the linear transport equation \eqref{Eq} constitutes an essential preliminary step for studying more complex coupled systems. Over the past few decades, significant progress has been made on the well-posedness of transport equations with weakly differentiable (in the space variables) velocity fields. In their seminal work, DiPerna and Lions \cite{MR1022305} established existence and uniqueness of solutions in $L^\infty$ under the assumption that $v$ belongs to $W^{1,1}$ (with bounded divergence).  Ambrosio \cite{MR2096794} later extended the result to BV vector fields, under additional conditions on the divergence. Chemin \cite{MR1688875} established a priori estimates for the transport equation in H\"{o}lder spaces $\mathcal{C}^r$, under the condition that the velocity field $v$ is divergence-free. A systematic well-posedness theory for \eqref{Eq} in Besov spaces has been developed in \cite{MR2768550,MR2231013}, provided the vector field $v$ is at least Lipschitz in the spatial variable.
	
	As is well known, the Triebel-Lizorkin spaces $F^s_{p,q}$ and their homogeneous versions $\dot{F}^s_{p,q}$ provide a unified framework encompassing many classical function spaces arising in the theory of partial differential equations. For instance, some notable equivalent characterizations include (see the details in \cite{MR2768550,MR781540}):
	\begin{itemize}[left=0pt]
		\item The Sobolev spaces $W^{s,p}=F^s_{p,2}$ for $s\in\mathbb{R},1<p<\infty.$ In particular, $H^s=F^s_{2,2}$ for $s\in\mathbb{R}$.
		\item The H\"{o}lder-Zygmund spaces $\mathscr{C}^s=F^s_{\infty,\infty}$ for $s>0. $
		\item The Hardy spaces $\mathcal{H}_p=\dot{F}^0_{p,2}$ for $0<p<\infty$. 
		\item The space of functions of bounded mean oscillation $BMO=\dot{F}^0_{\infty,2}$. 
		\item The Besov spaces: $B^s_{p,\min(p,q)} \hookrightarrow F^s_{p,q}\hookrightarrow B^s_{p,\max(p,q)}$ $(s\in\mathbb{R},(p,q)\in[1,\infty)\times [1,\infty]$ or $p=q=\infty$). In particular, $F^s_{p,p}=B^s_{p,p}$ $(s\in\mathbb{R},1\leq p\leq \infty)$.
	\end{itemize}
	This unifying perspective is a primary motivation for developing a dedicated theory within this setting. The theory of transport equations in Triebel-Lizorkin spaces, however, presents distinct and considerable challenges compared to that in Besov spaces. A fundamental technical reason lies in the non-commutativity of norms: for Triebel-Lizorkin spaces, the $L^p$ norm is taken after the $l^q$ norm over frequency blocks (see definition \eqref{qiciTLkongjiandingyi}), unlike in Besov spaces where the order is reversed. Consequently, the frequency pieces $\Delta_jf$ (defined in \eqref{Delta_jdingyi}) cannot be treated individually, they must be estimated simultaneously as a whole. This intrinsic structure makes tools like the vector-valued maximal function inequality (Lemma \ref{Fguji}) not merely convenient but virtually indispensable for establishing key estimates.
	
	More precisely, in the framework of Triebel-Lizorkin spaces, we develop the following general theory\footnote{Our results (Theorems \ref{xianyanguji} and \ref{shuyunfangchengcunzaixing}) for $x\in\mathbb{R}^d$ may be easily carried out to the periodic case $x\in\mathbb{T}^d$.} 
	(Theorems \ref{xianyanguji} and \ref{shuyunfangchengcunzaixing}) for the transport equation \eqref{Eq}, which seems to be new in the existing literature.
	
	\begin{theorem}(A priori estimates)\label{xianyanguji}
		Let $(p,q)\in[1,\infty)\times[1,\infty]$ or $p=q=\infty$. Suppose that $s>0$ and $v$ is a vector filed such that $\nabla v\in L^1(0,T;L^\infty(\mathbb{R}^d))$. Assume that $f_0\in F^s_{p,q}(\mathbb{R}^d), g\in L^1(0,T;F^s_{p,q}(\mathbb{R}^d))$ and that $f\in L^\infty(0,T;F^s_{p,q}(\mathbb{R}^d))\cap C([0,T];\mathscr{S}'(\mathbb{R}^d))$ solves the transport equation \eqref{Eq}. Then 
		\begin{enumerate}[leftmargin=22pt,label={(\roman*)}]
			\item if moreover, $v\in L^1(0,T;F^s_{p,q}(\mathbb{R}^d))$, $\nabla f\in L^\infty (0,T;L^\infty(\mathbb{R}^d))$ ,then there exists a constant $C$ depending only on $d,p,q$ and $s$, such that the following estimates holds for all $t\in[0,T]$,
			\begin{align}
				\|f\|_{F^s_{p,q}}\leq\|f_0\|_{F^s_{p,q}}+\int_{0}^{t}\|g\|_{F^s_{p,q}}d\tau+C\int_{0}^{t}\big(\|\nabla v\|_{L^\infty}\|f\|_{F^s_{p,q}}+\|\nabla f\|_{L^\infty}\|v\|_{F^s_{p,q}}\big)d\tau.\label{feiqiciF-priori-estimates}
			\end{align}
			\item If $s>1+\frac{d}{p}$ when $p>1$ or $s\geq 1+d$ when $p=1$ and $\nabla v \in L^1(0,T;F^{s-1}_{p,q}(\mathbb{R}^d))$, we have
			\begin{equation}\label{xianyangujijifenshizi}
				\|f\|_{F^s_{p,q}}\leq \|f_0\|_{F^s_{p,q}}+\int_{0}^{t}\|g\|_{F^s_{p,q}}d\tau+C\int_{0}^{t}U(\tau)\|f\|_{F^s_{p,q}}d\tau,
			\end{equation}
			or hence
			\begin{equation}\label{xianyangujishizi}
				\|f\|_{F^s_{p,q}}\leq e^{C\int_{0}^{t}\|\nabla v\|_{F^{s-1}_{p,q}}d\tau}\left(\|f_0\|_{F^s_{p,q}}+\int_{0}^{t}\|g\|_{F^s_{p,q}}e^{-C\int^\tau_0U(\tau')d\tau'}d\tau\right)
			\end{equation}
			with $U(t)\triangleq\|\nabla v\|_{F^{s-1}_{p,q}}$.
			\item Moreover, if $f=v$ and $\nabla v\in L^1(0,T;L^\infty)$, then for all $s>0$, the estimates \eqref{xianyangujijifenshizi} and \eqref{xianyangujishizi} hold with $U(t)=\|\nabla v\|_{L^\infty}$.
			\item If $\textup{div}v=0$ and $\nabla v\in L^1(0,T;F^s_{p,q}(\mathbb{R}^d))$, then for all $s>-1$, we have
			\begin{equation}
				\|f\|_{F^s_{p,q}}\leq\|f_0\|_{F^s_{p,q}}+\int_{0}^{t}\|g\|_{F^s_{p,q}}d\tau+C\int_{0}^{t}\big(\|\nabla v\|_{L^\infty}\|f\|_{F^s_{p,q}}+\| f\|_{L^\infty}\|\nabla v\|_{F^s_{p,q}}\big)d\tau.\label{feiqiciF-priori-estimates-divv=0}
			\end{equation}
		\end{enumerate}			
	\end{theorem}
	\begin{remark}
		A key technical distinction arises between the Besov and Triebel-Lizorkin spaces when applying the Littlewood-Paley method to establish a priori estimates. To control the solution in scales of function spaces, one naturally applies the frequency localization operator $\Delta_j$ to \eqref{Eq}. The process yields the term $[v,\Delta_j]\cdot\nabla f$ and $v\cdot\nabla\Delta_j f$. In Besov spaces, one directly handles $v\cdot\nabla\Delta_j f$ via standard energy arguments and integration by parts. This is infeasible in the Triebel-Lizorkin spaces due to their non-commutative norm structure ($l^q$ before $L^p$). Our proof therefore employs a characteristic method. Let us introduce particle trajectory mapping $X(t,\alpha)$, by definition, the solution to the following ordinary differential equation:
		\[ 
		\begin{cases}
			\partial_tX(t,\alpha)=v(t,X(t,\alpha)),\\
			X(0,\alpha)=\alpha.
		\end{cases} \]
		We use the following identity
		\begin{equation*}
			\partial_t\left(\Delta_jf\left(t,X(t,\alpha)\right)\right)=\Delta_jg(t,X(t,\alpha))+[v,\Delta_j]\cdot\nabla f(t,X(t,\alpha))
		\end{equation*}
		to absorb the problematic term $v\cdot\nabla\Delta_jf$ from the outset, as it is incorporated into the total derivative. The subsequent estimates then focus on controlling the commutator (see Proposition \ref{zijixiedejiaohuanzi} below) and the evolution of $X(t,\alpha)$ in $F^s_{p,q}$, which is better suited to its $l^q$-before-$L^p$ structure. This approach provides a natural path to the required a priori estimate without relying on the integration by parts technique that is central to the case of Besov spaces.
	\end{remark}
	\begin{theorem}\label{shuyunfangchengcunzaixing}
		(Local well-posedness for transport equation) Let $(p,q)\in[1,\infty)\times[1,\infty]$ or $p=q=\infty$. Suppose that $s>1+\frac{d}{p}$ when $p>1$ or $s\geq 1+d$ when $p=1$, $f_0 \in F^s_{p,q}(\mathbb{R}^d)$ and $g\in L^1(0,T;F^s_{p,q}(\mathbb{R}^d))$. Let v be a time dependent vector field with coefficients in $L^\rho(0,T;F^{-M}_{\infty,\infty}(\mathbb{R}^d))$ for some $\rho>1$ and $M>0$, and such that $\nabla v\in L^1(0,T;F^{s-1}_{p,q}(\mathbb{R}^d))$. Then the equation \eqref{Eq} has a unique solution $f\in L^\infty([0,T];F^s_{p,q}(\mathbb{R}^d))\bigcap(\cap_{s'<s}C([0,T];F^{s'}_{p,1}(\mathbb{R}^d))) $ and the estimates in Theorem \ref{xianyanguji} hold true. If moreover $q<\infty$, then we have $f\in C([0,T];F^s_{p,q}(\mathbb{R}^d))$. Furthermore, the data-to-solution map $f_0\mapsto f$ is continuous from $F^s_{p,q}$ into $L^\infty ([0,T];F^{s}_{p,q}(\mathbb{R}^d))\cap C([0,T],F^{s'}_{p,1}(\mathbb{R}^d))$ for every $s'<s$ if $q=\infty$, and into $C ([0,T];F^{s}_{p,q}(\mathbb{R}^d))$ if $q<\infty$.
	\end{theorem}
	\begin{remark}
		The existence of solution in Theorem \ref{shuyunfangchengcunzaixing} is established via compactness arguments applied to a sequence of regularized  approximate solutions $\{f^n\}_{n\in\mathbb{N}}$ which are the solutions to the transport equation with smoothed data $(f^n_0,g^n,v^n)$, obtained through frequency truncation and mollification.  To the best of our knowledge, a detailed implementation of such compactness method in the Triebel-Lizorkin spaces $F^s_{p,q}$ appears to be new, whereas it is standard in the context of Besov spaces (cf. \textup{\cite{MR2768550}}). Besides, a key step consists in establishing the time continuity of the solution, i.e. that $f\in C([0,T];F^s_{p,q})$ for $q<\infty$. This is equivalent to proving the uniform-in-time decay of the high-frequency part:
		\begin{equation*}
			\sup_{t\in[0,T]}\|f-S_nf\|_{F^s_{p,q}}\to0,\quad\text{as } n\to\infty.
		\end{equation*}
		The estimate for this term relies on controlling the commutator $[v,\Delta_j]\cdot \nabla f$, which constitutes the main difficulty due to the non‑commutative norm structure of $F^s_{p,q}$. This challenge is overcome by establishing a refined commutator estimate (see Remark \ref{s>n-proposition}).
	\end{remark}
	As mentioned above, in order to prove Theorems \ref{xianyanguji} and \ref{shuyunfangchengcunzaixing}, it is crucial to establish the following commutator estimates.
	\begin{proposition}(Commutator estimates)
		\label{zijixiedejiaohuanzi}	    	
		Let $(p,q)\in[1,\infty)\times[1,\infty]$, or $p=q=\infty$.
		Then for $s>0$, we have
		\begin{align}
			\Big\|\big\|2^{js}([f,\Delta_j]\cdot \nabla g) \big\|_{l^q(j\in\mathbb{Z})}\Big\|_{L^p}&\lesssim\sum_{j\in\mathbb{Z}}2^{js}\chi_{\{j\leq 4\}}\left(||\nabla f||_{L^\infty}||g||_{\dot{F}^s_{p,q}}+||\nabla g||_{L^\infty}||f||_{\dot{F}^s_{p,q}}\right)\label{qudiaodiv=0+young}\\
			& \lesssim\left(||\nabla f||_{L^\infty}||g||_{\dot{F}^s_{p,q}}+||\nabla g||_{L^\infty}||f||_{\dot{F}^s_{p,q}}\right)\label{qudiaodiv=0},
		\end{align}
		or for $s>-1$ and $\textup{div}f=0$,
		\begin{align}
			\Big\|\big\|2^{js}([f,\Delta_j]\cdot \nabla g) \big\|_{l^q(j\in\mathbb{Z})}\Big\|_{L^p} &\lesssim\sum_{j\in\mathbb{Z}}2^{js}\chi_{\{j\leq 4\}}\left(||\nabla f||_{L^\infty}||g||_{\dot{F}^s_{p,q}}+|| g||_{L^\infty}||\nabla f||_{\dot{F}^s_{p,q}}\right)\label{divf=0jiaohanzi+young}\\
			&\lesssim\left(||\nabla f||_{L^\infty}||g||_{\dot{F}^s_{p,q}}+|| g||_{L^\infty}||\nabla f||_{\dot{F}^s_{p,q}}\right),\label{divf=0jiaohanzi}
		\end{align}
		where the notation $A\lesssim B$ as an equivalent to $A \leq CB$ with some constant $C$.
	\end{proposition}
	\begin{remark}
		In our commutator estimate, we significantly relax the conditions required in \textup{\cite{MR2592288,MR4240785}}. Specifically, we extend the admissible range of parameters from $(p,q)\in(1,\infty)\times(1,\infty]$ or $p=q=\infty$ (see \textup{\cite{MR2592288}}) or $p=1,q\in[1,\infty]$ (see \textup{\cite{MR4240785}}) to the whole $(p,q)\in[1,\infty)\times[1,\infty]$ or $p=q=\infty$ and, more importantly, remove the condition $\textup{div}f=0$ for $s>0$ (see the following picture), which admits extension to broader systems, in particular to the compressible case.

	
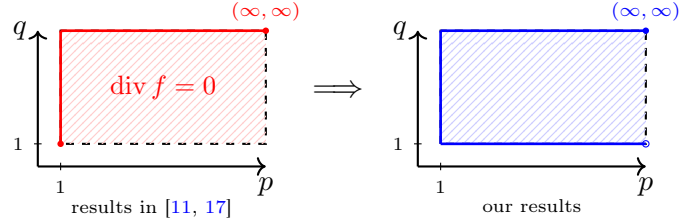
\vspace{-1.0em} 
	\[ \begin{tikzpicture}[
		scale=1,
		point/.style={circle, fill=black, inner sep=0.6pt}
		]
		% 第一部分：原始条件
		\begin{scope}[local bounding box=original]
			% 坐标轴
			\draw[->, thick] (0,0) -- (3.0,0) node[below] {$p$};
			\draw[->, thick] (0,0) -- (0,1.8) node[left] {$q$};
			
			% 刻度
			\draw (0.3,0.05) -- (0.3,-0.05) node[below, font=\tiny] {$1$};
			\draw (0.05,0.3) -- (-0.05,0.3) node[left, font=\tiny] {$1$};
			
			% 原始条件区域
			\draw[dashed, thick] (0.3,0.3) rectangle (3.0,1.8);
			\fill[pattern=north east lines, pattern color=red!20] 
			(0.3,0.3) rectangle (3.0,1.8);
			
			% 在区域中心添加文字 "div f=0"
			\node[red, align=center, font=\small\\] at (1.65, 1.05) {$\mathrm{div}\, f=0$};
			
			% 线段: p = 1, q ∈ [1,∞] (用实线表示)
			\draw[red, line width=1pt] (0.3,0.3) -- (0.3,1.8);
			\draw[red, line width=1pt] (0.3,1.8) -- (3.0,1.8);
			
			% 点: p = q = 1 (不标注，只用点表示)
			\fill[red] (0.3,0.3) circle (1.2pt);
			
			% 点: p = q = ∞
			\fill[red] (3.0,1.8) circle (1.2pt);
			\node[red, above, font=\tiny] at (3.0,1.8) {$(\infty,\infty)$};
			
			% 标题放在坐标轴下方
			\node[below, align=center, font=\tiny] at (1.5,-0.3) 
			{\textup{results in \cite{MR2592288,MR4240785}}};
		\end{scope}
		
		% 第二部分：拓展后条件
		\begin{scope}[local bounding box=extended, xshift=5cm]
			% 坐标轴
			\draw[->, thick] (0,0) -- (3.0,0) node[below] {$p$};
			\draw[->, thick] (0,0) -- (0,1.8) node[left] {$q$};
			
			% 刻度
			\draw (0.3,0.05) -- (0.3,-0.05) node[below, font=\tiny] {$1$};
			\draw (0.05,0.3) -- (-0.05,0.3) node[left, font=\tiny] {$1$};
			
			\draw[dashed, thick] (0.3,0.3) rectangle (3.0,1.8);
			\fill[pattern=north east lines, pattern color=blue!20] 
			(0.3,0.3) rectangle (3.0,1.8);
			
			% 拓展后条件区域
			\draw[blue, line width=1pt] (0.3,0.3) -- (3.0,0.3); % 下边
			\draw[blue, line width=1pt] (0.3,0.3) -- (0.3,1.8); % 左边
			\draw[blue, line width=1pt] (0.3,1.8) -- (3.0,1.8); % 上边
			
			% 点: p=q=∞
			\fill[blue] (3.0,1.8) circle (1.2pt);
			\draw[blue, fill=none, thin] (3.0,0.3) circle (1.2pt); 
			\node[blue, above, font=\tiny] at (3.0,1.8) {$(\infty,\infty)$};
			
			% 标题放在坐标轴下方
			\node[below, align=center, font=\tiny] at (1.5,-0.3) 
			{\textup{our results}};
		\end{scope}
		
		% 连接两部分的箭头 - 使用数学符号\Rightarrow
		\node[font=\large] at (3.95, 1.0) {$\Longrightarrow$};
		
	\end{tikzpicture} \]
	More precisely, based on the Hardy-Littlewood maximal function properties (see Lemmas \ref{miaoguji} and \ref{guoguji}), when estimating the commutator $[ f,\Delta_j]\cdot\nabla g$ for $s>0$ without assuming $\text{div} f = 0$, an extra term emerges during integration by parts. The main innovation here lies in establishing a precise estimate for this term, thereby eliminating the need for the divergence-free assumption and extending the existing theory to a broader class of problems.
	\end{remark} 
	 
	\subsection{Applications.}
	Note that the ideal MHD system can be recast in the form of transport equations via a suitable transformation. It is this well-known structural property that motivates the second main objective of this paper: to apply the general theory for the transport equation (Theorems \ref{xianyanguji} and \ref{shuyunfangchengcunzaixing}) to study the Cauchy problem for the incompressible ideal MHD system. We aim to revisit and extend the well-posedness theory in the framework of Triebel-Lizorkin spaces. Specifically, we provide a complete proof of local well-posedness in the sense of Hadamard, which, in addition to existence and uniqueness, includes the continuous dependence of the solution on the initial data—a crucial aspect not addressed in the prior work of  Chen, Miao and Zhang \cite{MR2592288} . Our approach, via the unified transport theory, not only offers a distinct perspective but also streamlines the proof of these results. Furthermore, we derive a precise blow-up criterion for this system. The versatility of our framework suggests its applicability to other systems, such as the Euler-Poincaré equations, shallow water wave equations, etc., which we plan to explore in a forthcoming paper. 
	
	The ideal MHD equations in $\mathbb{R}^d$ is given as follows:
	\begin{equation}\label{MHD}
		\begin{cases}
			u_t+u\cdot\nabla u=-\nabla p-\dfrac{1}{2}\nabla b^2+b\cdot\nabla b,\\
			b_t+u\cdot\nabla b=b\cdot\nabla u,\\
			\nabla\cdot u=\nabla \cdot b=0,\\
			u(0,x)=u_0(x), \quad b(0,x)=b_0(x),
		\end{cases}\tag{IMHD}
	\end{equation}
	where $t\in\mathbb{R}^+,x\in\mathbb{R}^d$, $u$ and $b$ denote the flow velocity vector and the magnetic field vector, respectively, $p$ is a scalar pressure, while $u_0$ and $b_0$ are the given initial data satisfying $\nabla\cdot u_0=\nabla\cdot b_0=0$. If we set
	\begin{equation*}
		z^+=u+b,\qquad z^-=u-b,
	\end{equation*}
	then \eqref{MHD} can be rewritten to the system for $z^+$ and $z^-$ as follows:
	\begin{equation}\label{yongMHD}
		\begin{cases}
			\partial_tz^++(z^-\cdot\nabla)z^+=-\nabla\pi,\\
			\partial_tz^-+(z^+\cdot\nabla)z^-=-\nabla\pi,\\
			\nabla\cdot z^+=\nabla  \cdot z^-=0,\\
			z^+(0)=z^+_0=u_0+b_0,\quad z^-(0)=z^-_0=u_0-b_0,			
		\end{cases}
	\end{equation}
	where $\pi=p+\frac{1}{2}b^2$. The well-posedness theory for the ideal MHD system has been extensively studied across various function spaces. For initial data $(u_0,b_0)\in H^s(\mathbb{R}^d)\times H^s(\mathbb{R}^d)$ with $ s>1+\frac{d}{2}$, the standard energy method \cite{MR748308} yields a unique local smooth solution $(u,b)\in C([0,T];H^s\times H^s)\cap C^1([0,T];H^{s-1}\times H^{s-1}).$ However, the question of whether such local solutions can be extended globally or will develop finite-time singularities remains an outstanding open problem. In the classical Sobolev spaces $W^{s,p}$, Alexseev \cite{MR752597} established fundamental existence and uniqueness results. In \cite{MR2275873}, Miao and Yuan prove that there exists a locally unique solution in the critical Besov space $B^{1+d/p}_{p,1}(\mathbb{R}^d)$ for $1\leq p\leq \infty$ provided that the initial data $(u_0,b_0)$ is in this critical space. Note that $B^{s-1}_{p,q}(\mathbb{R}^d)$ is a Banach algebra for $s>1+\frac{d}{p}$. One can easily prove that there exists a unique smooth solution $(u,b) \in C([0,T];B^s_{p,q}\times B^s_{p,q})\cap C^1([0,T];B^{s-1}_{p,q}\times B^{s-1}_{p,q})$ to \eqref{MHD} by standard method, see \cite{MR1664597} for details. In the framework of Triebel-Lizorkin spaces $F^s_{p,q}$, it is shown in \cite{MR2592288} that for $s>1+\frac{d}{p}$	and $1<p,q<\infty$, the system admits a unique solution $(u,b)\in C([0,T];F^s_{p,q}\times F^s_{p,q})$. 
	
	A central theme in this analysis is the derivation of blow-up criteria, which give conditions under which smooth solutions persist. The celebrated Beale-Kato-Majda blow-up criterion for the incompressible Euler equations \cite{MR763762} was extended to the ideal MHD system by Caflisch et al. in \cite{MR1462753}. More precisely, they showed that if the smooth solution $(u,b)\in C([0,T];H^s\times H^s )\cap C^1([0,T];H^{s-1} \times H^{s-1})$ satisfies the following condition in $H^s$ with $s\geq 3$:
	\begin{equation*}
		\int_{0}^T\big(\|\nabla\times u\|_{L^\infty}+\|\nabla\times b\|_{L^\infty}\big)dt<\infty,
	\end{equation*}
	then the solution $(u,b)$ can be extended beyond $t=T$. In the framework of Triebel-Lizorkin spaces, Chen, Miao and Zhang \cite{MR2592288} established that solutions stay smooth as long as
	\begin{equation*}
		\int_{0}^T\big(\|\nabla\times u\|_{\dot{F}^0_{\infty,\infty}}+\|\nabla\times b\|_{\dot{F}^0_{\infty,\infty}}\big)dt<\infty.
	\end{equation*}
	We refer to \cite{MR2299432,MR2173645} for the other refined blow-up criteria.
	
	Building upon this established landscape, the second main objective of this paper is to apply the general theory for transport equations (Theorems \ref{xianyanguji} and \ref{shuyunfangchengcunzaixing}) to establish the local well-posedness for the system \eqref{MHD} in the sense of Hadamard and to derive a corresponding blow-up criterion in the Triebel-Lizorkin spaces $F^s_{p,q}$. A principal difficulty arises in handling the critical regularity regimes, and this challenge is overcome by employing key technical lemmas (see Lemmas \ref{p=1Reszi} and \ref{guoguji}). Before stating our local well-posedness result, we introduce the following function spaces:
	\[ E^s_{p,q}(T)\triangleq
	\begin{cases}{}
		C([0,T];F^s_{p,q}(\mathbb{R}^d))\cap C^1([0,T];F^{s-1}_{p,q}(\mathbb{R}^d))	& \text{if } q< \infty,\\
		L^\infty([0,T];F^s_{p,\infty}(\mathbb{R}^d))\cap \text{Lip}([0,T];F^{s-1}_{p,\infty}(\mathbb{R}^d))  &\text{if } q=\infty.
	\end{cases}
	\]
	with $T>0,s\in\mathbb{R}$, and $(p,q)\in[1,\infty)\times[1,\infty]$ or $p=q=\infty$. Our second main results (Theorems \ref{MHD-Local-well-posedness} and \ref{Blow-up Criterion}) are then stated as follows.
	\begin{theorem}\label{MHD-Local-well-posedness}
		Let $(p,q)\in[1,\infty)\times[1,\infty]$ or $p=q=\infty$, $s>1+\frac{d}{p}$ when $p>1$ or $s\geq 1+d$ when $p=1$, and  $(u_0,b_0)\in F^s_{p,q}(\mathbb{R}^d)\times F^s_{p,q}(\mathbb{R}^d) $ satisfying $\textup{div}u_0=\textup{div}b_0=0$. Then there exists a time $T>0$ such that \eqref{MHD} has a unique solution $(u,b)\in E^s_{p,q}(T)\times E^s_{p,q}(T)$. Furthermore, the data-to-solution map $(u_0,b_0)\mapsto (u,b)$ is continuous from $D(R)\times D(R) $ into $ C([0,T],F^{s'}_{p,q}\times F^{s'}_{p,q})$ for every $s'<s$ if $q=\infty$, and $s=s'$ if $q<\infty$, where $D(R)\triangleq\{f\in F^s_{p,q}: \|f\|_{F^{s}_{p,q}}\leq R,\textup{div}f=0\}$.
	\end{theorem}
		\begin{theorem}\label{Blow-up Criterion}
		The local-in-time solution $(u,b)\in C([0,T];F^s_{p,q}(\mathbb{R}^d)\times F^s_{p,q}(\mathbb{R}^d))$ was constructed in Theorem \ref{MHD-Local-well-posedness}. If the lifespan of the solution $T^*$ is finite, then
		\begin{equation}\label{blow-up1}
			\int_{0}^{T^*}\big(\|u\|_{L^\infty}+\|b\|_{L^\infty}+\|\nabla u\|_{L^\infty}+\|\nabla b\|_{L^\infty}\big)dt=+\infty.
		\end{equation}
		Moreover, 
		\begin{enumerate}[label={(\roman*)}]
			\item if $(p,q)\in[1,\infty)\times[1,\infty]$ or $p=q=\infty$ and $s>1+\frac{d}{p}$, then
			\begin{equation}\label{blow-up2}
				\int_{0}^{T^*}\big(\| u\|_{L^\infty}+\|b\|_{L^\infty}+\|(\nabla\times u)(t)\|_{\dot{F}^0_{\infty,\infty}}+\|(\nabla\times b)(t)\|_{\dot{F}^0_{\infty,\infty}}\big)dt=+\infty.
			\end{equation}
			\item if $(p,q)\in(1,\infty)\times[1,\infty]$ and $s>1+\frac{d}{p}$, then
			\begin{equation}\label{blow-up3}
				\int_{0}^{T^*}\big(\|(\nabla\times u)(t)\|_{\dot{F}^0_{\infty,\infty}}+\|(\nabla\times b)(t)\|_{\dot{F}^0_{\infty,\infty}}\big)dt=+\infty.
			\end{equation}
		\end{enumerate}
	\end{theorem}
	\begin{remark}
		Theorem \ref{MHD-Local-well-posedness} provides a more complete picture than earlier results. It refines and extends Theorem 1.1 of \textup{\cite{MR2592288}} by simultaneously encompassing both 
		the subcritical case ($1<p,q<\infty$) treated in \textup{\cite{MR2592288}}
		 and the critical case ($p=1,q\in[1,\infty]$). Furthermore, it establishes the continuous dependence of the solutions on the initial data, a property not addressed in the aforementioned work.
	\end{remark}
	\begin{remark}
		As mentioned before, $W^{s,p}=F^s_{p,2}(1<p<\infty,s\in\mathbb{R})$  and $\mathscr{C}^s=F^s_{\infty,\infty}(s>0)$. So, our Theorem \ref{MHD-Local-well-posedness} implies the local well-posedness for \eqref{MHD} in the Sobolev spaces and H\"{o}lder-Zygmund spaces. Besides, Theorem \ref{Blow-up Criterion} covers and generalizes the blow-up criterion results in \textup{\cite{MR1462753,MR2299432,MR2592288,MR2173645}}.
	\end{remark}
	\begin{remark}
		Note that when the magnetic field $b=0$, \eqref{MHD} reduces to the incompressible Euler equations. Thus, Theorems \ref{MHD-Local-well-posedness} and \ref{Blow-up Criterion} provides local well-posedness and blow-up criterion for the Euler equations in the Triebel-Lizorkin spaces $F^s_{p,q}$. Our results cover the classical well-posedness theory in \textup{\cite{MR1880646,MR2020259,MR1688875,MR4240785,MR951744}} and encompass key blow-up criteria in \textup{\cite{MR763762,MR1880646,MR1980623,MR1794270}}, the latter connection is clarified via the embedding relations $L^\infty\hookrightarrow BMO\hookrightarrow \dot{F}^0_{\infty,\infty}=\dot{B}^0_{\infty,\infty}$. 
	\end{remark}
	
	The remainder of the paper is organized as follows. In Section 2, we recall some facts
	on the Littlewood-Paley analysis, Triebel-Lizorkin spaces, Hardy-Littlewood maximal functions, and Morse type inequalities as well. In Section 3, we develop commutator estimates in the homogeneous Triebel-Lizorkin spaces. In Section 4, we establish a priori estimates (Theorem \ref{xianyanguji}) for the transport equations \eqref{Eq}. In Section 5, we prove the local well-posedness (Theorem \ref{shuyunfangchengcunzaixing}) for transport equations in the Triebel-Lizorkin spaces. Finally, in Section 6, we apply Theorems \ref{xianyanguji} and \ref{shuyunfangchengcunzaixing} from Sections 4 and 5 to \eqref{MHD} or system \eqref{yongMHD}, yielding its local well-posedness (Theorem \ref{MHD-Local-well-posedness}) and blow-up criteria (Theorem \ref{Blow-up Criterion}).

	\section{Preliminaries}\label{Preliminaries}
	We introduce the Littlewood-Paley decomposition. Let $\mathscr{B}=\{\xi\in\mathbb{R}^d\colon|\xi|\leq \frac{4}{3} \}$ and $\mathscr{C}=\{\xi\in\mathbb{R}^d\colon\frac{3}{4}\leq |\xi|\leq \frac{8}{3} \}$, $(\varphi,\chi)$ be a couple of smooth functions valued in $[0,1]$, such that $\varphi$ is support in  $\mathscr{C}$, $\chi$ is supported in $\mathscr{B}$ and 
	\[ \forall\, \xi \in \mathbb{R}^d,\quad \chi(\xi)+\sum_{j\in\mathbb{N}}\varphi(2^{-j}\xi)=1, \]
	\[ \forall\, \xi \in \mathbb{R}^d\setminus\{0\},\quad \sum_{j\in\mathbb{Z}}\varphi(2^{-j}\xi)=1. \]
	We denote $\varphi_j(\xi)=\varphi(2^{-j}\xi),h=\mathscr{F}^{-1}\varphi$ and $\tilde{h}=\mathscr{F}^{-1}\chi$. Denote $\dot{\mathscr{S}}'(\mathbb{R}^d)$ as the dual space of $\dot{\mathscr{S}}(\mathbb{R}^d)\triangleq\{f\in \mathscr{S}(\mathbb{R}^d)\colon\partial^\alpha\hat{f}(0)=0,\forall\text{ multi-index }\alpha\in\mathbb{N}^d \}$ with the Schwartz space $\mathscr{S}(\mathbb{R}^d)$. For $f\in\dot{\mathscr{S}}'(\mathbb{R}^d)$, one can define dyadic blocks as follows:
	\begin{equation}\label{Delta_jdingyi}
		\Delta_jf=\varphi(2^{-j}D)f=2^{jd}\int_{\mathbb{R}^d}h(2^jy)f(x-y)dy,
	\end{equation}
	\begin{equation}\label{S_jdingyi}
		S_jf=\sum_{k\leq j-1}\Delta_kf=\chi(2^{-j}D)f=2^{jd}\int_{\mathbb{R}^d}\tilde{h}(2^jy)f(x-y)dy.
	\end{equation}
	It is easy to see $\Delta_j=S_j-S_{j-1}$. For all $f,g\in \dot{\mathscr{S}}'(\mathbb{R}^d)$, we observe
	\begin{equation}\label{a.o.c}
		\Delta_j\Delta_kf\equiv 0\quad \text{if} \quad |j-k|\geq 2\quad\text{and}\quad\Delta_j(\Delta_{k-1}f\Delta_kg)\equiv0\quad\text{if}\quad k\leq j-1,
	\end{equation}
	\begin{equation}\label{a.o.c-1}
		 S_j\Delta_kf\equiv0\quad\text{if}\quad j\leq k-1\quad\text{and}\quad \Delta_j(S_{k-1}f\Delta_kg)\equiv 0\quad \text{if} \quad |j-k|\geq 5.
	\end{equation}
	
	\begin{remark}\label{remark1}
		For each $j\in\mathbb{Z}$, we have $\|\Delta_jf\|_{L^p},\|S_jf\|_{L^p}\leq C\|f\|_{L^p}$ for some positive constant C independent of $j$.
	\end{remark}
	\begin{lemma}\label{Bernstein}
		\textup{\cite{MR2768550}}(Bernstein's inequalities) Let $\mathcal{B}$ be a ball and $\mathcal{C}$ an annulus. A constant $C>0$ exists such that for any $k\in\mathbb{N},1\leq p\leq q\leq \infty$, and any function $f\in L^p(\mathbb{R}^d)$, we have
		\[ \text{Supp}\hat{f}\subset\lambda\mathcal{B}\Rightarrow\|D^kf\|_{L^q}\triangleq\sup_{|\alpha|=k}\|\partial^\alpha f\|_{L^q}\leq C^{k+1}\lambda  ^{k+d(\frac{1}{p}-\frac{1}{q})}\|f\|_{L^p}, \] 
		\[ \text{Supp}\hat{f}\subset\lambda\mathcal{C}\Rightarrow C^{-k-1}\lambda^k\|f\|_{L^p}\leq\|D^kf\|_{L^p}\leq C^{k+1}\lambda  ^{k}\|f\|_{L^p}. \] 
	\end{lemma}
	\begin{lemma}\label{p=1Reszi}
		\textup{\cite{MR4240785,MR1232192}} Let $m(\xi)$ be the Fourier symbol of the operator $S_{0}(-\Delta)^{-1}\allowbreak\partial_l\partial_k$, $1\leq k,l\leq d$, then there exists a constant C, such that
		\begin{equation*}
\|\mathscr{F}^{-1}\left(m(\xi)\xi_i\right)\|_{L^1(\mathbb{R}^d)}\leq C,\quad \forall\text{ } 1\leq i\leq d.
		\end{equation*}
	\end{lemma} 
	Let us recall the definition of Triebel-Lizorkin spaces \cite{MR781540,MR1163193}. Let $s\in\mathbb{R},(p,q)\in [1,\infty)\times[1,\infty]$ or $p=q=\infty$. The homogeneous Triebel-Lizorkin space $\dot{F}^s_{p,q}(\mathbb{R}^d)$\footnote{For the endpoint case $p=q=\infty$, $\dot{F}^s_{\infty,\infty}$ is equipped with the norm
		\[ \|f\|_{\dot{F}^s_{\infty,\infty}}=\sup_{j\in \mathbb{Z}}2^{js}\|\Delta_jf\|_{L^\infty}, \]
		which coincides with the Besov spaces $\dot{B}^s_{\infty,\infty}$.} ($\dot{F}^s_{p,q}$ for short) is defined by 
	\begin{equation*}
		\dot{F}^s_{p,q}(\mathbb{R}^d)=\{f\in \dot{\mathscr{S}}'(\mathbb{R}^d)\colon \|f\|_{\dot{F}^s_{p,q}}<\infty\},
	\end{equation*}
	where
	\begin{equation}\label{qiciTLkongjiandingyi}
		\|f\|_{\dot{F}^s_{p,q}}=\Big\|\big\|2^{js}|\Delta_jf|\big\|_{l^q(j\in\mathbb{Z})}\Big\|_{L^p}.
	\end{equation}
	For $s>0$, $(p,q)\in [1,\infty)\times[1,\infty]$ or $p=q=\infty$, we define the inhomogeneous Triebel-Lizorkin space $F^s_{p,q}(\mathbb{R}^d)$ ($F^s_{p,q}$ for short) as follows:
	\begin{equation*}
		F^s_{p,q}(\mathbb{R}^d)=\{f\in \mathscr{S}'(\mathbb{R}^d)\colon \|f\|_{F^s_{p,q}}<\infty\},
	\end{equation*}
	where
	\begin{equation*}
		\|f\|_{F^s_{p,q}}=\|f\|_{L^p}+\|f\|_{\dot{F}^s_{p,q}}.
	\end{equation*}
	 The inhomogeneous Triebel-Lizorkin space is a Banach space equipped with the norm $\|\cdot\|_{F^s_{p,q}}$. 
	 
	 \begin{lemma}\label{lem:Triebel-Lizorkin-properties}
	 	\textup{\cite{MR1419319,MR781540}}The following properties for Triebel-Lizorkin spaces hold:
	 	\begin{enumerate}[label={(\roman*)}]
	 		\item Embedding: Let $s\in\mathbb{R},\varepsilon>0$ and suppose $q_0<q_1$. Then
	 		\[ F^s_{p,q_0}\hookrightarrow F^s_{p,q_1},\qquad F^{s+\varepsilon}_{p,q}\hookrightarrow F^s_{p,q}. \]
	 		\item Algebraic properties: For $s>0$, $F^s_{p,q}(\mathbb{R}^d)\cap L^\infty$ is an algebra. Moreover, $F^s_{p,q}(\mathbb{R}^d) $ is an algebra $\Longleftrightarrow$ $F^s_{p,q}(\mathbb{R}^d)\hookrightarrow L^\infty\Longleftrightarrow s>\frac{d}{p}$ or $s\geq d$ with $p=1$.
	 		\item Fatou property: If the sequence $\{f_k\}_{k\in\mathbb{N}_+}$ is uniformly bounded in $F^s_{p,q}$ and converges weakly in $\mathscr{S}'$ to $f$, then $f\in F^s_{p,q}$ and $\|f\|_{F^s_{p,q}}\leq \liminf\limits_{k\to\infty}\|f_k\|_{F^s_{p,q}}$.
	 		\item Complex interpolation: Let $1\leq p_0,q_0\leq \infty,1\leq p_1,q_1<\infty,0<\theta<1$ and
	 		\[ \frac{1}{q}=\frac{1-\theta}{q_0}+\frac{\theta}{q_1},\quad\frac{1}{p}=\frac{1-\theta}{p_0}+\frac{\theta}{p_1},\quad  s=(1-\theta)s_0+\theta s_1. \] 
	 		Then we have
	 		\[ \|f\|_{F^{s}_{p,q}}\leq\|f\|_{F^{s_0}_{p_0,q_0}}^{1-\theta}\|f\|_{F^{s_1}_{p_1,q_1}}^\theta,\quad\forall\ f\in F^{s_0}_{p_0,q_0}\cap F^{s_1}_{p_1,q_1}. \] 
	 	\end{enumerate}
	 \end{lemma}
	 \begin{lemma}\textup{\cite{MR1880646}}\label{lem:logarithmic-Triebel-Lizorkin-space-inequality}
	 	(Logarithmic interpolation inequality)
	 	Let $s> \frac{d}{p}$ with $(p,q)\in[1,\infty)\times[1,\infty]$ or $p=q=\infty$. Suppose $f\in F^s_{p,q}(\mathbb{R}^d)$, then there exists a positive constant C such that 
	 	\begin{equation*}
	 		\|f\|_{L^\infty}\leq C\big(1+\|f\|_{\dot{F}^0_{\infty,\infty}}(\ln^+\|f\|_{F^{s}_{p,q}}+1)\big),
	 	\end{equation*}
	 	where $\ln^+f\triangleq\max(0,\ln f)$.
	 \end{lemma}
	 \begin{lemma}\label{tiduheqiciF}
	 	\textup{\cite{MR781540}}For any $k\in\mathbb{N}$, there exists a constant $C_k$ such that the following inequality holds:
	 	\begin{equation}
	 		C_k^{-1}\|\nabla^kf\|_{\dot{F}^s_{p,q}}\leq\|f\|_{\dot{F}^{s+k}_{p,q}}\leq C_k\|\nabla^kf\|_{\dot{F}^s_{p,q}}.
	 	\end{equation}
	 \end{lemma}
	 \begin{lemma}\label{lem:Snguji}
	 	Let $s\in\mathbb{R}$, $(p,q)\in[1,\infty)\times[1,\infty]$ or $p=q=\infty$. Then there exits some $C>0$, such that 
	 	\[ \|S_{n+1}f\|_{F^{s+l}_{p,q}}\leq C2^{nl}\|f\|_{F^s_{p,q}}, \qquad\forall\  n,l\in\mathbb{N}.\]
	 \end{lemma}
	 \begin{proof}
	 	For $(p,q)\in[1,\infty)\times[1,\infty]$, the result is already established in \cite{MR4240785}. Thus, we only need to consider the case $p=q=\infty$. Since \eqref{a.o.c-1}, it follows from Remark \ref{remark1} that
	 	\[ \|S_{n+1}f\|_{\dot{F}^{s+l}_{\infty,\infty}}=\sup_{j\in\mathbb{Z}}2^{j(s+l)}\|\Delta_jS_{n+1}f\|_{L^\infty}
	 	\leq C2^{nl}\sup_{j\in\mathbb{Z}}2^{js}\|\Delta_jf\|_{L^\infty}
	 	= C2^{nl}\|f\|_{\dot{F}^s_{\infty,\infty}}. \]
	 \end{proof}
	 
	If $f$ is a locally Lebesgue integrable function on $\mathbb{R}^d$, then the Hardy-Littlewood maximal function of $f$ is defined as follows:
	\[ (Mf)(x)=\sup_{r>0}\frac{1}{|\mathscr{B}(x,r)|}\int_{\mathscr{B}(x,r)}|f(y)|dy, \]
	where $|\mathscr{B}(x,r)|$ is the volume of the ball $\mathscr{B}(x,r)$ with center $x$ and radius $r$.
	\begin{lemma}\label{Fguji}
		\textup{\cite{c6f44d5b-9c43-3ae5-9150-1b0704d03ff8,MR1232192}} Let $(p,q)\in(1,\infty)\times(1,\infty]$ or $p=q=\infty$. Suppose $\{f_j\}_{j\in\mathbb{Z}}$ is a sequence of functions in $L^p$ with the property that $||f_j(\cdot)||_{l_j^q}\in L^p(\mathbb{R}^d)$, then there exists some $C>0$, such that
		\[\Big\|\big\|Mf_j(\cdot)\big\|_{l^q_j}\Big\|_{L^p} \leq C \Big\|\big\|f_j(\cdot)\big\|_{l^q_j}\Big\|_{L^p}. \]
	\end{lemma}
	\begin{lemma}\label{miaoguji}
		\textup{\cite{Stein+1971}} Let $\varphi$ be an integrable function on $\mathbb{R}^d $, and set $\varphi_\varepsilon(x)=\frac{1}{\varepsilon^d}\varphi(\frac{x}{\varepsilon})$ for $\varepsilon>0$. Suppose that the least decreasing radial majorant of $\varphi$ is integrable, i.e. let \[ \psi(x)=\sup_{|y|\geq|x|}|\varphi(y)| \]and we suppose $\int_{\mathbb{R}^d}\psi(x)=A<\infty$. Then for all $f\in L^p(\mathbb{R}^d),1\leq p\leq\infty, $
		\[ \sup_{\varepsilon>0}|(f\ast\varphi_\varepsilon)(x)|\leq AM(f)(x),  \]
		where $M(f)$ is the Hardy-Littlewood maximal function of $f$.
	\end{lemma}
	\begin{lemma}\label{guoguji}
		\textup{\cite{MR4240785,MR781540}} Let $L>0,j,k\in\mathbb{Z},j>k-L$ and $r \in (0,\infty)$. $ \psi\in C^{\infty}(\mathbb{R}^d) $ satisfies
		\[ |\psi(z)|(1+|z|^{\frac{d}{r}})\leq g(z), \]
		where $g(z)$ is some nonnegative radial decreasing integrable function. Denote $\psi_k(x)=2^{kd}\psi(2^kx)$, then for any $\theta\in (0,1]$, there exists a constant $C$ independent of $j,k$, such that the following inequality
		\[ |(\psi_k\ast f)(x)|\leq C2^{(j-k)\theta\frac{d}{r}}M(|f|^{1-\theta})(x)[M(|f|^r)(x)]^{\frac{\theta}{r}}\]
		holds for all $f\in L_{\mathscr{B}(0,c2^j)}^p$ with $p\geq 1$ and some generic constant c, where $L_{\mathscr{B}(0,c2^j)}^p\triangleq\{f\in L^p\colon\textup{supp}\hat{f}\subset \mathscr{B}(0,c2^j)\} .$
	\end{lemma}
	Now, we recall the following Morse type inequalities.
	\begin{lemma}\label{lem:product}
		\textup{\cite{MR1419319}} Assume $s_1\leq s_2$ and $s_1+s_2>d \max(0,\frac{2}{p}-1).$ Let $s_2>\frac{d}{p}$ and $q\geq\max(q_1,q_2).$ In the case $s_2>s_1$, we have 
		\[  \|fg\|_{F^{s_1}_{p,q_1}(\mathbb{R}^d)}\lesssim \|f\|_{F^{s_1}_{p,q_1}(\mathbb{R}^d)}\|g\|_{F^{s_2}_{p,q_2}(\mathbb{R}^d)}.\]
		If $s_1=s_2$, then 
		\[ \|fg\|_{F^{s_1}_{p,q}(\mathbb{R}^d)}\lesssim \|f\|_{F^{s_1}_{p,q_1}(\mathbb{R}^d)}\|g\|_{F^{s_1}_{p,q_2}(\mathbb{R}^d)}.\]
	\end{lemma}
	\begin{proposition}\label{moser-type-estimate}
		Let $s>0$, $(p,q)\in[1,\infty)\times[1,\infty]$, or $p=q=\infty$. Then there exists a constant $C$ such that the following inequalities hold:
		\begin{equation}\label{homogeneous-Triebel-Lizorkin-space-moser}
			\|fg\|_{\dot{F}^s_{p,q}}\leq C\left(\|f\|_{L^{\infty}}\|g\|_{\dot{F}^s_{p,q}}+\|g\|_{L^{\infty}}\|f\|_{\dot{F}^s_{p,q}}\right)
		\end{equation}
		for the homogeneous Triebel-Lizorkin space, and
		\begin{equation}\label{inhomogeneous-Triebel-Lizorkin-space-moser}
			\|fg\|_{F^s_{p,q}}\leq C\left(\|f\|_{L^{\infty}}\|g\|_{F^s_{p,q}}+\|g\|_{L^{\infty}}\|f\|_{F^s_{p,q}}\right)
		\end{equation}
		for the inhomogeneous Triebel-Lizorkin space.
	\end{proposition}
	\begin{proof}
		In the case of $(p,q)\in(1,\infty)\times(1,\infty]$ or $p=q=\infty$, Proposition \ref{moser-type-estimate} was proved in \cite{MR1880646}. In the case of $p=1,q\in[1,\infty]$, Proposition \ref{moser-type-estimate} has been  established in \cite{MR2020259,MR4240785}. Hence, it suffices to show the case of $q=1,p\in(1,\infty)$. Indeed, we use Bony's paraproduct decomposition \cite{MR631751} as follows:
		\[ fg=T_fg+T_gf+R(f,g), \]
		where
		\[ T_fg=\sum_{k\leq j-2}\Delta_kf\Delta_jg=\sum_{j\in\mathbb{Z}}S_{j-1}f\Delta_jg, \]
		\[ R(f,g)=\sum_{j\in\mathbb{Z}}\Delta_jf\tilde{\Delta}_jg,\quad\tilde{\Delta}_j\coloneqq\Delta_{j-1}+\Delta_{j}+\Delta_{j+1}. \]
		For every $0<r<\infty$, by using Lemma \ref{guoguji} with $\theta=1$, one has
		\begin{align}
			\left|2^{ms}\Delta_mT_fg\right|&=\left|2^{ms}\sum_{j\in\mathbb{Z}}\Delta_m(S_{j-1}f\Delta_jg)\right|=\left|\sum_{|m-j|\leq 4}2^{ms}\Delta_m(S_{j-1}f\Delta_jg)\right|\nonumber\\
			&\leq C \sum_{|m-j|\leq 4}2^{ms}\left[M(|(S_{j-1}f\Delta_jg)|^r)\right]^{\frac{1}{r}}.\label{equation1}
		\end{align}
		As such, choosing $0<r<1$, and applying  \eqref{equation1}, Young's inequality and Lemma \ref{Fguji} yields
		\begin{align*}
			\|T_fg\|_{\dot{F}^s_{p,1}}&=\left\|\left\|2^{ms}\left|\Delta_mT_fg\right|\right\|_{l^1}\right\|_{L^p}\\
			&\leq C\Bigg\| \bigg\|\sum_{|m-j|\leq 4}2^{ms}\left[M(|(S_{j-1}f\Delta_jg)|^r)\right]^{\frac{1}{r}}\bigg\|_{l^1}\Bigg\|_{L^p}\\
			&\leq C\left\|f\right\|_{L^\infty}\Bigg\| \bigg\| \sum_{|m-j|\leq 4}2^{(m-j)s}2^{js}\left[M\left(|(\Delta_jg)|^r\right)\right]^{\frac{1}{r}}\bigg\|_{l^1}\Bigg\|_{L^p}\\
			&\leq C\left\|f\right\|_{L^\infty}\Bigg\| \bigg\| 2^{js}\left[M\left(|(\Delta_jg)|^r\right)\right]^{\frac{1}{r}}\bigg\|_{l^1}\Bigg\|_{L^p}\\
			&=C\left\|f\right\|_{L^\infty}\left\| \left\| M\left(2^{jsr}|(\Delta_jg)|^r\right)\right\|_{l^\frac{1}{r}}\right\|_{L^{\frac{p}{r}}}^{\frac{1}{r}}\\
			&\leq C\left\|f\right\|_{L^\infty}\left\| \left\| 2^{jsr}|(\Delta_jg)|^r\right\|_{l^\frac{1}{r}}\right\|_{L^{\frac{p}{r}}}^{\frac{1}{r}}\\
			&=C\|f\|_{L^\infty}\|g\|_{\dot{F}^s_{p,1}}.
		\end{align*}	    	
		Similarly, we have
		\[ \|T_gf\|_{\dot{F}^s_{p,1}}\leq C\|g\|_{L^\infty}\|f\|_{\dot{F}^s_{p,1}}. \]
		Next, we estimate $R(f,g)=\sum_{j\in\mathbb{Z} }\Delta_jf\tilde{\Delta}_jg$. For arbitrary fixed $r\in(0,1)$, as $s>0$, we can specify $\theta\in(0,1)$ such that $s>\frac{d\theta}{r}$. Using the property of Fourier frequency support \eqref{a.o.c}, one can assert that $\Delta_m\left(\Delta_jf\tilde{\Delta}_jg\equiv0\right)$ if $j<m-3$. Thanks to Lemma \ref{guoguji}, one gets    	
		\begin{align}
			\left|2^{ms}\Delta_mR(f,g)\right|&=\left|2^{ms}\Delta_m\sum_{j\in\mathbb{Z} }\Delta_jf\tilde{\Delta}_jg\right|=\left|2^{ms}\sum_{j\geq m-3}\Delta_m\left(\Delta_jf\tilde{\Delta}_jg\right)\right|\nonumber\\
			&\leq C\Bigg| 2^{ms}\sum_{j\geq m-3}2^{(j-m)\theta\frac{d}{r}}M(|\Delta_jf\tilde{\Delta}_jg|^{1-\theta})(x) \times [M(|\Delta_jf\tilde{\Delta}_jg|^r)(x)]^\frac{\theta}{r}\Bigg|\nonumber\\
			&= C\Bigg| \sum_{j\geq m-3}2^{(m-j)(s-\theta\frac{d}{r})}M(|2^{js}\Delta_jf\tilde{\Delta}_jg|^{1-\theta})(x)\times [M(|2^{js}\Delta_jf\tilde{\Delta}_jg|^r)(x)]^\frac{\theta}{r}\Bigg|.\label{equation2}
		\end{align}	    	
		Following Young's inequality and H\"{o}lder's inequality, Lemma \ref{Fguji}, we can see	 from \eqref{equation2} that	
		\begin{align}
			\|R(f,g)\|_{\dot{F}^s_{p,1}}&=\Big\|\big\|2^{ms}\left|\Delta_mR(f,g)\right|\big\|_{l^1}\Big\|_{L^p}\nonumber\\
			&\leq C\Bigg\|\bigg\|M(|2^{js}\Delta_jf\tilde{\Delta}_jg|^{1-\theta})(x) [M(|2^{js}\Delta_jf\tilde{\Delta}_jg|^r)(x)]^\frac{\theta}{r}\bigg\|_{l^1}\Bigg\|_{L^p}\nonumber\\
			&\leq C\Bigg\|\bigg\|M(|2^{js}\Delta_jf\tilde{\Delta}_jg|^{1-\theta})(x)\bigg\|_{l^\frac{1}{1-\theta}}\Bigg\|_{L^\frac{p}{1-\theta}}\times\Bigg\|\bigg\| [M(|2^{js}\Delta_jf\tilde{\Delta}_jg|^r)(x)]^\frac{\theta}{r}\bigg\|_{l^\frac{1}{\theta}}\Bigg\|_{L^\frac{p}{\theta}}\nonumber\\
			&\leq C\Bigg\|\bigg\||2^{js}\Delta_jf\tilde{\Delta}_jg|^{1-\theta}\bigg\|_{l^\frac{1}{1-\theta}}\Bigg\|_{L^\frac{p}{1-\theta}}\times\Bigg\|\bigg\| [M(|2^{js}\Delta_jf\tilde{\Delta}_jg|^r)(x)]^\frac{\theta}{r}\bigg\|_{l^\frac{1}{\theta}}\Bigg\|_{L^\frac{p}{\theta}}\nonumber\\
			&\leq C\|f\|_{L^\infty}^{1-\theta}\|g\|_{\dot{F}^s_{p,1}}^{1-\theta}\Bigg\|\bigg\| M(|2^{js}\Delta_jf\tilde{\Delta}_jg|^r)(x)\bigg\|_{l^\frac{1}{r}}\Bigg\|_{L^\frac{p}{r}}^{\frac{\theta}{r}}\nonumber\\
			&\leq C\|f\|_{L^\infty}^{1-\theta}\|g\|_{\dot{F}^s_{p,1}}^{1-\theta}\left\|\left\| |2^{js}\Delta_jf\tilde{\Delta}_jg|^r\right\|_{l^\frac{1}{r}}\right\|_{L^\frac{p}{r}}^{\frac{\theta}{r}}\nonumber\\
			&\leq C\|f\|_{L^\infty}^{1-\theta}\|g\|_{\dot{F}^s_{p,1}}^{1-\theta} \|f\|_{L^\infty}^{\theta}\|g\|_{\dot{F}^s_{p,1}}^{\theta}\nonumber\\
			&= C\|f\|_{L^\infty}\|g\|_{\dot{F}^s_{p,1}}\label{moserzhongRbufen}.
		\end{align}   
		This yields \eqref{homogeneous-Triebel-Lizorkin-space-moser}. The  inequality \eqref{inhomogeneous-Triebel-Lizorkin-space-moser} for the inhomogeneous space is obtained by \eqref{homogeneous-Triebel-Lizorkin-space-moser} and the following fact:
		\begin{equation*}
			\|fg\|_{L^p}\leq \frac{1}{2}(\|f\|_{L^\infty}\|g\|_{L^p}+\|g\|_{L^\infty}\|f\|_{L^p}).
		\end{equation*}
		Therefore, we complete the proof of Proposition \ref{moser-type-estimate}.		
	\end{proof}
	\section{Commutator estimates} 
	This section is devoted to the derivation of the commutator estimates Proposition \ref{zijixiedejiaohuanzi}.
	\begin{proof}[Proof of Proposition \ref{zijixiedejiaohuanzi}]
		 By the Einstein convention on the summation over repeated indices $i\in\{1,\cdots,d\}$, and the Bony paraproduct decomposition, one can see
	\begin{align*}
		[f,\Delta_j]\cdot \nabla g&=f^i\Delta_j\partial_ig-\Delta_j(f^i\partial_ig)\\
		&=[T_{f^i},\Delta_j]\partial_ig-\Delta_jR(f^i,\partial_ig)+T_{\Delta_j\partial_ig}f^i+R(f^i,\Delta_j\partial_ig)-\Delta_jT_{\partial_ig}f^i\\
		&\triangleq I_1+I_2+I_3+I_4+I_5.
	\end{align*} 	    	
	\textbf{Case 1:} $\bm{(p,q) \in (1,\infty)\times(1,\infty]}$ \textbf{or} $\bm{p=q=\infty.}$
	
	Using \eqref{a.o.c-1}, integration by part, first-order Taylor's formula and Lemma \ref{miaoguji}, we have
	\begin{align*}
	|I_1|&=\left|\sum_{m\in\mathbb{Z}}S_{m-1}f^i\Delta_m\Delta_j\partial_ig-\Delta_j(\sum_{m\in\mathbb{Z}}S_{m-1}f^i\Delta_m\partial_ig)\right|\\
		&=\Bigg|\sum_{|m-j|\leq 4}S_{m-1}f^i(x)\Delta_m2^{jd}\int_{\mathbb{R}^d}h\left(2^j(x-y)\right)\partial_ig(y)dy\\
		&\mathrel{\phantom{=}}-2^{jd}\int_{\mathbb{R}^d}h\left(2^j(x-y)\right)\left(S_{m-1}f^i(y)\Delta_m\partial_ig(y)\right)dy\Bigg|\\
		&=\left|\sum_{|m-j|\leq 4}\int_{\mathbb{R}^d}\left(S_{m-1}f^i(x)-S_{m-1}f^i(y)\right)2^{jd}h\left(2^j(x-y)\right)\partial_i\Delta_mg(y)dy\right|\\
		&=\Bigg|\sum_{|m-j|\leq 4}\int_{\mathbb{R}^d}-\left(\partial_iS_{m-1}f^i(x)-\partial_iS_{m-1}f^i(y)\right)2^{jd}h\left(2^j(x-y)\right)\Delta_mg(y)\\
		&\mathrel{\phantom{=}}-\left(S_{m-1}f^i(x)-S_{m-1}f^i(y)\right)2^{jd+j}(\partial_ih)\left(2^j(x-y)\right)\Delta_mg(y)dy\Bigg|\\
		&\lesssim\left\|\nabla f\right\|_{L^\infty}\sum_{|m-j|\leq 4}\left|\int_{\mathbb{R}^d}2^{jd}h\left(2^j(x-y)\right)\Delta_mg(y)dy\right|\\
		&\mathrel{\phantom{=}}+\left\|\nabla S_{m-1}f\right\|_{L^\infty}\sum_{|m-j|\leq 4}\left|\int_{\mathbb{R}^d}2^{jd+j}|x-y|(\partial_ih)\left(2^j(x-y)\right)\Delta_mg(y)dy\right|\\
		&\lesssim\sum_{|m-j|\leq 4}\|\nabla f\|_{L^\infty}M(|\Delta_mg|)(x)+\sum_{|m-j|\leq 4}\|\nabla S_{m-1} f\|_{L^\infty}M(|\Delta_mg|)(x)\\
		&\lesssim\sum_{|m-j|\leq 4}\|\nabla f\|_{L^\infty}M(|\Delta_mg|)(x).	   	
	\end{align*}
	Thanks to Young's inequality and Lemma \ref{Fguji}, one infers
	\begin{align}
		\Big\|\big\|2^{js}|I_1|\big\|_{l^q}\Big\|_{L^p}
		&\lesssim\|\nabla f\|_{L^\infty}\Bigg\|\bigg\|2^{js}\sum_{|m-j|\leq 4}M(|\Delta_mg|)(x)	\bigg\|_{l^q}\Bigg\|_{L^p}\nonumber\\
		&=\|\nabla f\|_{L^\infty}\Bigg\|\bigg\|\sum_{|m-j|\leq 4}2^{(j-m)s}M(2^{ms}|\Delta_mg|)(x)	\bigg\|_{l^q}\Bigg\|_{L^p}\nonumber\\
		&\lesssim\|\nabla f\|_{L^\infty}\sum_{j\in\mathbb{Z}}2^{js}\chi_{\{|j|\leq 4\}}\Big\|\big\|M(2^{ms}|\Delta_mg|)(x)	\big\|_{l^q}\Big\|_{L^p}\nonumber\\
		&\lesssim\sum_{j\in\mathbb{Z}}2^{js}\chi_{\{|j|\leq 4\}}\|\nabla f\|_{L^\infty}\Big\|\big\|2^{ms}|\Delta_mg|	\big\|_{l^q}\Big\|_{L^p}\nonumber\\
		&\lesssim\sum_{j\in\mathbb{Z}}2^{js}\chi_{\{|j|\leq 4\}}\|\nabla f\|_{L^\infty}\|g\|_{\dot{F}^s_{p,q}}\label{I_3+Young}\\
		&\lesssim\|\nabla f\|_{L^\infty}\|g\|_{\dot{F}^s_{p,q}}\label{I_3}.
	\end{align}
	
	For the term $I_2$, by applying integration by parts, first-order Taylor formula, and invoking Lemma \ref{miaoguji}, we have
	\begin{align}
		|I_2|&=\left|\sum_{m\in\mathbb{Z}}\Delta_j\left(\Delta_mf^i\tilde{\Delta}_m\partial_ig\right)\right|=\left|\sum_{m\geq j-3}\Delta_j\left(\Delta_mf^i\tilde{\Delta}_m\partial_ig\right)\right|\nonumber\\
		&=\left|\sum_{m\geq j-3}2^{jd}\int_{\mathbb{R}^d}h\left(2^j(x-y)\right)\left(\Delta_mf^i(y)\tilde{\Delta}_m\partial_ig(y)\right)dy\right|\nonumber\\
		&\lesssim\left|\sum_{m\geq j-3}2^{jd+j}\int_{\mathbb{R}^d}(\partial_ih)\left(2^j(x-y)\right)\Delta_mf^i(y)\tilde{\Delta}_mg(y)dy\right|\nonumber\\
		&\mathrel{\phantom{=}}+\left|\sum_{m\geq j-3}2^{jd}\int_{\mathbb{R}^d}h\left(2^j(x-y)\right)\partial_i\Delta_mf^i(y)\tilde{\Delta}_mg(y)dy\right|\label{I_5fenbujifenhou}\\
		&\lesssim\sum_{m\geq j-3}2^jM(\tilde{\Delta}_mg)(x)\|\Delta_mf\|_{L^\infty}+\sum_{m\geq j-3}M(\tilde{\Delta}_mg)(x)\|\nabla\Delta_mf\|_{L^\infty}.\nonumber
	\end{align}
	Now applying Lemma \ref{Bernstein}, Young's inequality and Lemma \ref{Fguji}, for $s>0$, one can find
	\begin{align}
		\Big\|\big\|2^{js}|I_2|\big\|_{l^q}\Big\|_{L^p}
		&\lesssim\Bigg\|\bigg\|\sum_{m\geq j-3}2^{js+j}M(\tilde{\Delta}_mg)(x)\|\Delta_mf\|_{L^\infty}\bigg\|_{l^q}\Bigg\|_{L^p}\nonumber\\
		&\mathrel{\phantom{=}}+\Bigg\|\bigg\|\sum_{m\geq j-3}2^{js}M(\tilde{\Delta}_mg)(x)\|\nabla\Delta_mf\|_{L^\infty}\bigg\|_{l^q}\Bigg\|_{L^p}\nonumber\\
		&\lesssim\Bigg\|\bigg\|\sum_{m\geq j-3}2^{(j-m)(s+1)}M(2^{ms}\tilde{\Delta}_mg)(x)\|\nabla\Delta_mf\|_{L^\infty}\bigg\|_{l^q}\Bigg\|_{L^p}\nonumber\\
		&\mathrel{\phantom{=}}+\Bigg\|\bigg\|\sum_{m\geq j-3}2^{(j-m)s}M(2^{ms}\tilde{\Delta}_mg)(x)\|\nabla\Delta_mf\|_{L^\infty}\bigg\|_{l^q}\Bigg\|_{L^p}\nonumber\\
		&\lesssim\sum_{j\in\mathbb{Z}}2^{j(s+1)}\chi_{\{j\leq 3\}}\Big\|\big\|M(2^{ms}\tilde{\Delta}_mg)(x)\|\nabla\Delta_mf\|_{L^\infty}\big\|_{l^q}\Big\|_{L^p}\nonumber\\
		&\mathrel{\phantom{=}}+\sum_{j\in\mathbb{Z}}2^{js}\chi_{\{j\leq 3\}}\Big\|\big\|M(2^{ms}\tilde{\Delta}_mg)(x)\|\nabla\Delta_mf\|_{L^\infty}\big\|_{l^q}\Big\|_{L^p}\nonumber\\
		&\lesssim\sum_{j\in\mathbb{Z}}2^{js}\chi_{\{j\leq 3\}}\|\nabla f\|_{L^\infty}\Big\|\big\|M(2^{ms}\tilde{\Delta}_mg)(x)\big\|_{l^q}\Big\|_{L^p}\nonumber\\
		&\lesssim\sum_{j\in\mathbb{Z}}2^{js}\chi_{\{j\leq 3\}}\|\nabla f\|_{L^\infty}\Big\|\big\|2^{ms}|\tilde{\Delta}_mg|\big\|_{l^q}\Big\|_{L^p}\nonumber\\
		&\lesssim\sum_{j\in\mathbb{Z}}2^{js}\chi_{\{j\leq 3\}}\|\nabla f\|_{L^\infty}\|g\|_{\dot{F}^s_{p,q}}\label{I_5+young}\\
		&\lesssim\|\nabla f\|_{L^\infty}\|g\|_{\dot{F}^s_{p,q}}\label{I_5}.
	\end{align}
	We note that if $\text{div}f=0 $, then the second term in \eqref{I_5fenbujifenhou}, which results from integration by parts, is identically zero. It then follows from the above argument that for $s>-1$ and $\text{div }f=0$,
	\begin{align}
		\Big\|\big\|2^{js}|I_2|\big\|_{l^q}\Big\|_{L^p}&\lesssim\sum_{j\in\mathbb{Z}}2^{j(s+1)}\chi_{\{j\leq 3\}}\|\nabla f\|_{L^\infty}\|g\|_{\dot{F}^s_{p,q}}\label{I_5+young-divf=0}\\
		&\lesssim\|\nabla f\|_{L^\infty}\|g\|_{\dot{F}^s_{p,q}}\label{I_5-divf=0}.
	\end{align}
	The remaining terms $I_3,I_4,I_5$ can be handled via (4.9) and (4.11) in \cite{MR2592288}, which yields 
	\begin{align}
		\Big\|\big\|2^{js}|I_3+I_4|\big\|_{l^q}\Big\|_{L^p}&\lesssim\sum_{j\in\mathbb{Z}}2^{js}\chi_{\{j\leq 2\}}\|\nabla g\|_{L^\infty}\|f\|_{\dot{F}^s_{p,q}}\label{I_1+I_2+young}\\
		&\lesssim\|\nabla g\|_{L^\infty}\|f\|_{\dot{F}^s_{p,q}},\label{I_1+I_2}
	\end{align}
	and
	\begin{align}
		\Big\|\big\|2^{js}|I_5|\big\|_{l^q}\Big\|_{L^p}&\lesssim\sum_{j\in\mathbb{Z}}2^{js}\chi_{\{|j|\leq 4\}}\|\nabla g\|_{L^\infty}\|f\|_{\dot{F}^s_{p,q}}\label{I_4+young}\\
		&\lesssim\|\nabla g\|_{L^\infty}\|f\|_{\dot{F}^s_{p,q}}.\label{I_4}
	\end{align}
	Combining \eqref{I_3+Young}, \eqref{I_5+young}, \eqref{I_1+I_2+young}, and \eqref{I_4+young}, we obtain, in this case, \eqref{qudiaodiv=0+young}. Similarly, a combination of $\eqref{I_3}$, \eqref{I_5}, \eqref{I_1+I_2} and \eqref{I_4} yields \eqref{qudiaodiv=0} for the present scenario.\\
	\textbf{Case 2: }$\bm{p=1}$ \textbf{with} $\bm{q\in[1,\infty]}$	\textbf{or} $\bm{q=1}$ \textbf{with} $\bm{p\in(1,\infty).}$
	
	Concerning the term $I_1$, we first note that 
	\begin{align*}
		|I_1|&=\left|\sum_{|m-j|\leq 4}\int \left(S_{m-1}f^i(x)-S_{m-1}f^i(y)\right)
		2^{jd}h\left(2^j\left(x-y\right)\right)\Delta_m\partial_ig(y)dy\right|\\
		&\lesssim\left|\sum_{|m-j|\leq 4}\int \left(S_{m-1}f^i(x)-S_{m-1}f^i(y)\right)
		2^{j(d+1)}(\partial_ih)\left(2^j\left(x-y\right)\right)\Delta_mg(y)dy\right|\\
		&\mathrel{\phantom{=}}+\left|\sum_{|m-j|\leq 4}\int \left(S_{m-1}\partial_if^i(x)-S_{m-1}\partial_if^i(y)\right)
		2^{jd}h\left(2^j\left(x-y\right)\right)\Delta_mg(y)dy\right|\\
		&\lesssim\left\|\nabla S_{m-1}f\right\|_{L^\infty
		}\left|\sum_{|m-j|\leq 4}\int\left|x-y\right| 
		2^{j(d+1)}\nabla h\left(2^j\left(x-y\right)\right)\Delta_mg(y)dy\right|\\
		&\mathrel{\phantom{=}}+\|\nabla f\|_{L^\infty}\left|\sum_{|m-j|\leq 4}\int 
		2^{jd} h\left(2^j\left(x-y\right)\right)\Delta_mg(y)dy\right|\\
		&\lesssim\sum_{|m-j|\leq 4}\left\|\nabla f\right\|_{L^\infty}\left[M(|\Delta_mg|^r)(x)\right]^{\frac{1}{r}
		},
	\end{align*}
	here we used first-order Taylor's formula and Lemma \ref{guoguji} with $\theta=1$ and $r\in(0,1)$. Therefore, it follows from Young's inequality and  Lemma \ref{Fguji} that
	\begin{align}
		\Big\|\big\|2^{js}|I_1|\big\|_{l^q}\Big\|_{L^p} 
		&\lesssim \|\nabla f\|_{L^\infty} \Bigg\|\bigg\|\sum_{|m-j|\leq 4} 2^{(j-m)s} \left[M(|2^{ms}\Delta_m g|^r)(x)\right]^{\frac{1}{r}}\bigg\|_{l^q}\Bigg\|_{L^p} \nonumber\\
		&\lesssim \|\nabla f\|_{L^\infty}\sum_{j\in\mathbb{Z}}2^{js}\chi_{\{|j|\leq 4\}}\Big\|\big\|[M(|2^{ms}\Delta_mg|^r)(x)]^{\frac{1}{r}}\big\|_{l^q}\Big\|_{L^p}\nonumber\\
		&=\sum_{j\in\mathbb{Z}}2^{js}\chi_{\{|j|\leq 4\}}\|\nabla f\|_{L^\infty}\Big\|\big\|[M(|2^{ms}\Delta_mg|^r)(x)]\big\|_{l^{\frac{q}{r}}}\Big\|_{L^{\frac{p}{r}}}^{\frac{1}{r}}\nonumber\\
		&\lesssim\sum_{j\in\mathbb{Z}}2^{js}\chi_{\{|j|\leq 4\}}\|\nabla f\|_{L^\infty}\Big\|\big\|2^{ms}|\Delta_mg|^r(x)\big\|_{l^{\frac{q}{r}}}\Big\|_{L^{\frac{p}{r}}}^{\frac{1}{r}}\nonumber\\
		&\lesssim \sum_{j\in\mathbb{Z}}2^{js}\chi_{\{|j|\leq 4\}}\|\nabla f\|_{L^\infty}\|g\|_{\dot{F}^s_{p,q}}\label{endpoint-I_3+young}\\
		&\lesssim\|\nabla f\|_{L^\infty}\|g\|_{\dot{F}^s_{p,q}}\label{endpoint-I_3}.
	\end{align}
	
	Regarding the term $I_2$, since $s>0$, for arbitrary $r\in (0,1)$, one can select $\theta \in (0,1)$ small enough, such that $s>d\frac{\theta}{r}$. Due to frequency interaction \eqref{a.o.c}, one observes that $\Delta_j(\Delta_mf\tilde{\Delta}_mg)\equiv0$ if $m <j-3$. Using  integration by parts and Lemma \ref{guoguji}, we deduce
	\begin{align*}
		|I_2|
		&=\left|\Delta_j\sum_{m\in\mathbb{Z}}\Delta_mf^i\tilde{\Delta}_m\partial_ig\right|=\left|\sum_{m\geq j-3}\Delta_j\left(\Delta_mf^i\tilde{\Delta}_m\partial_ig\right)\right|\\
		&=\left|\sum_{m\geq j-3}\int2^{jd}h\left(2^j(x-y)\right)\Delta_mf^i(y)\tilde{\Delta}_m\partial_ig(y)dy\right|\\
		&\lesssim\left|\sum_{m\geq j-3}\int2^{jd+j}(\partial_ih)\left(2^j(x-y)\right)\Delta_mf^i(y)\tilde{\Delta}_mg(y)dy\right|\\
		&\mathrel{\phantom{=}}+\left|\sum_{m\geq j-3}\int2^{jd}h\left(2^j(x-y)\right)\Delta_m\partial_if^i(y)\tilde{\Delta}_mg(y)dy\right|\\
		&\lesssim\sum_{m\geq j-3}2^j2^{(m-j)\frac{d\theta}{r}}M(|\Delta_mf^i\tilde{\Delta}_mg|^{1-\theta})\left[M(|\Delta_mf^i\tilde{\Delta}_mg|^r)(x)\right]^{\frac{\theta}{r}}\\
		&\mathrel{\phantom{=}}+\sum_{m\geq j-3}2^{(m-j)\frac{d\theta}{r}}M(|\Delta_m\partial_if^i\tilde{\Delta}_mg|^{1-\theta})\left[M(|\Delta_m\partial_if^i\tilde{\Delta}_mg|^r)(x)\right]^{\frac{\theta}{r}}.
	\end{align*}    	
	Thanks to Young's inequality, H\"{o}lder's inequality and Lemma \ref{Bernstein}, one can argue analogously as \eqref{moserzhongRbufen} to infer
	\begin{align}
		&\mathrel{\phantom{=}}\Big\|\big\|2^{js}|I_2|\big\|_{l^q}\Big\|_{L^p}\nonumber\\
		&\lesssim\Bigg\|\bigg\|\sum_{m\geq j-3}2^{(j-m)(s+1-\frac{d\theta}{r})}M(|(2^m\Delta_mf^i)(2^{ms}\tilde{\Delta}_mg)|^{1-\theta})\left[M(|(2^m\Delta_mf^i)(2^{ms}\tilde{\Delta}_mg)|^r)(x)\right]^{\frac{\theta}{r}}\bigg\|_{l^q}\Bigg\|_{L^p}\nonumber\\
		&\mathrel{\phantom{=}}+\Bigg\|\bigg\|\sum_{m\geq j-3}2^{(j-m)(s-\frac{d\theta}{r})}M(|(\Delta_m\partial_if^i)(2^{ms}\tilde{\Delta}_mg)|^{1-\theta})\left[M(|(\Delta_m\partial_if^i)(2^{ms}\tilde{\Delta}_mg)|^r)(x)\right]^{\frac{\theta}{r}}\bigg\|_{l^q}\Bigg\|_{L^p}\nonumber\\
		&\lesssim\sum_{j\in\mathbb{Z}}2^{j(s+1-\frac{d\theta}{r})}\chi_{\{j\leq 3\}}\Bigg\|\bigg\|M(|(2^m\Delta_mf^i)(2^{ms}\tilde{\Delta}_mg)|^{1-\theta})\left[M(|(2^m\Delta_mf^i)(2^{ms}\tilde{\Delta}_mg)|^r)(x)\right]^{\frac{\theta}{r}}\bigg\|_{l^q}\Bigg\|_{L^p}\nonumber\\
		&\mathrel{\phantom{=}}+\sum_{j\in\mathbb{Z}}2^{j(s-\frac{d\theta}{r})}\chi_{\{j\leq 3\}}\Bigg\|\bigg\|M(|(\Delta_m\partial_if^i)(2^{ms}\tilde{\Delta}_mg)|^{1-\theta})\left[M(|(\Delta_m\partial_if^i)(2^{ms}\tilde{\Delta}_mg)|^r)(x)\right]^{\frac{\theta}{r}}\bigg\|_{l^q}\Bigg\|_{L^p}\nonumber\\
		&\lesssim\sum_{j\in\mathbb{Z}}2^{j(s-\frac{d\theta}{r})}\chi_{\{j\leq 3\}}\|\nabla f\|_{L^\infty}\Big\|\big\|M(|2^{ms}\tilde{\Delta}_mg|^{1-\theta})(x)\big[M(2^{ms}\tilde{\Delta}_mg)^r(x)\big]^{\frac{\theta}{r}}\big\|_{l^q}\Big\|_{L^p}\nonumber\\
		&\lesssim \sum_{j\in\mathbb{Z}}2^{j(s-\frac{d\theta}{r})}\chi_{\{j\leq 3\}}\|\nabla f\|_{L^\infty}\|g\|_{\dot{F}^s_{p,q}}\label{endpoint-I_5+young}\\
		&\lesssim \|\nabla f\|_{L^\infty}\|g\|_{\dot{F}^s_{p,q}}\label{endpoint-I_5}.
	\end{align}
	Analogous to \eqref{I_5+young-divf=0} and \eqref{I_5-divf=0}, we find that if $\text{div}f=0$, then for any $s>-1$, we have
	\begin{align}
		\Big\|\big\|2^{js}|I_2|\big\|_{l^q}\Big\|_{L^p}&\lesssim\sum_{j\in\mathbb{Z}}2^{j(s+1-\frac{d\theta}{r})}\chi_{\{j\leq 3\}}\|\nabla f\|_{L^\infty}\|g\|_{\dot{F}^s_{p,q}}\label{endpoint-I_5+young-divf=0}\\
		&\lesssim\|\nabla f\|_{L^\infty}\|g\|_{\dot{F}^s_{p,q}}\label{endpoint-I_5-divf=0}.
	\end{align}
	
	Note that $S_{m+2}\Delta_jf\equiv0$ if $m\leq j-3$, thus
	\begin{align*}
		I_3+I_4
		&=\sum_{m\in\mathbb{Z}}\left(S_{m-1}\Delta_j  \partial_ig\right)\Delta_mf^i+\sum_{m\in\mathbb{Z}}(\Delta_mf^i)(\tilde{\Delta}_m\Delta_j\partial_ig)\\
		&=\sum_{m\in\mathbb{Z}}\left(S_{m+2}\Delta_j  \partial_ig\right)\Delta_mf^i=\sum_{m\geq j-2}\left(S_{m+2}\partial_i\Delta_j  g\right)\Delta_mf^i.
	\end{align*}
	Then thanks to Young's inequality, we get for $s>0$,
	\begin{align}
		\Big\|\big\|2^{js}|I_3+I_4|\big\|_{l^q}\Big\|_{L^p}
		&=\Bigg\|\bigg\|\sum_{m\geq j-2 }\left(2^{(j-m)s}S_{m+2}\partial_i\Delta_j  g\right)2^{ms}\Delta_mf^i \bigg\|_{l^q}\Bigg\|_{L^p}\nonumber\\
		&\lesssim\left\|\nabla g\right\|_{L^\infty}\Bigg\|\bigg\|\sum_{m\geq j-2 }2^{(j-m)s}2^{ms}\Delta_mf \bigg\|_{l^q}\Bigg\|_{L^p}\nonumber\\
		&\lesssim\left\|\nabla g\right\|_{L^\infty}\sum_{j\in\mathbb{Z}}2^{js}\chi_{\{j\leq 2\}}\Big\|\big\|2^{ms}\Delta_mf \big\|_{l^q}\Big\|_{L^p}\nonumber\\
		&\lesssim\sum_{j\in\mathbb{Z}}2^{js}\chi_{\{j\leq 2\}}\left\|\nabla g\right\|_{L^\infty}\|f\|_{\dot{F}^s_{p,q}}\label{endpoint-I_1+I_2+young}\\
		&\lesssim\left\|\nabla g\right\|_{L^\infty}\|f\|_{\dot{F}^s_{p,q}}\label{endpoint-I_1+I_2}.
	\end{align}	    	

	Finally, we estimate the term $I_5$. By \eqref{a.o.c-1}, Lemma \ref{guoguji} with $\theta=1$ and $r\in(0,1)$, Young's inequality and using Lemma \ref{Fguji} as above, we have	    	
	\begin{align}
		\Big\|\big\|2^{js}|I_5|\big\|_{l^q}\Big\|_{L^p}
		&=\Bigg\|\bigg\|2^{js} \sum_{|m-j|\leq 4}\Delta_j\left(S_{m-1}\partial_ig\Delta_mf^i\right)\bigg\|_{l^q}\Bigg\|_{L^p}\nonumber\\
		&\lesssim\Bigg\|\bigg\|2^{js} \sum_{|m-j|\leq 4}\left[M(|S_{m-1}\partial_ig\Delta_mf^i|^r)(x)\right]^\frac{1}{r}\bigg\|_{l^q}\Bigg\|_{L^p}\nonumber\\
		&=\Bigg\|\bigg\| \sum_{|m-j|\leq 4}2^{(j-m)s}\left[M(|S_{m-1}\partial_ig2^{ms}\Delta_mf^i|^r)(x)\right]^\frac{1}{r}\bigg\|_{l^q}\Bigg\|_{L^p}\nonumber\\
		&\lesssim \sum_{j\in\mathbb{Z}}2^{js}\chi_{\{|j|\leq 4\}}\Big\|\big\| \left[M(|S_{m-1}\partial_ig2^{ms}\Delta_mf^i|^r)(x)\right]^\frac{1}{r}\big\|_{l^q}\Big\|_{L^p}\nonumber\\
		&\lesssim\sum_{j\in\mathbb{Z}}2^{js}\chi_{\{|j|\leq 4\}}\|\nabla g\|_{L^\infty}\|f\|_{\dot{F}^s_{p,q}}\label{endpoint-I_4+young}\\
		&\lesssim\|\nabla g\|_{L^\infty}\|f\|_{\dot{F}^s_{p,q}}\label{endpoint-I_4}.
	\end{align}    		    	 
	
	Summing up \eqref{endpoint-I_3+young}, \eqref{endpoint-I_5+young}, \eqref{endpoint-I_1+I_2+young} and \eqref{endpoint-I_4+young} we reach \eqref{qudiaodiv=0+young}. Likewise, By combining \eqref{endpoint-I_3}, \eqref{endpoint-I_5}, \eqref{endpoint-I_1+I_2} and \eqref{endpoint-I_4} we derive \eqref{qudiaodiv=0}.
	
	In order to show \eqref{divf=0jiaohanzi+young} and \eqref{divf=0jiaohanzi}, in view of Proposition 4.1 in \cite{MR2592288} and Proposition 2.12 in \cite{MR4240785}, it suffices to prove \eqref{divf=0jiaohanzi} in the case of $p\in(1,\infty)$ and $q
	=1$ with $\text{div}f$=0. Indeed, we only need to modify the estimates in $I_3+I_4$ and $I_5$ as follows. Note that when $s>-1$ and $\text{div}f=0$,
	\begin{align*}
		\Big\|\big\|2^{js}|I_3+I_4|\big\|_{l^1}\Big\|_{L^p}
		&=\Bigg\|\bigg\|\sum_{m\geq j-2 }\left(2^{(j-m)s}S_{m+2}\partial_i\Delta_j  g\right)2^{ms}\Delta_mf^i \bigg\|_{l^1}\Bigg\|_{L^p}\nonumber\\
		&\lesssim\left\|\nabla\Delta_j g\right\|_{L^\infty}\Bigg\|\bigg\|\sum_{m\geq j-2 }2^{(j-m)s}2^{ms}\Delta_mf \bigg\|_{l^1}\Bigg\|_{L^p}\nonumber\\
		&\lesssim\left\|\Delta_jg\right\|_{L^\infty}\Bigg\|\bigg\|\sum_{m\geq j-2 }2^{(j-m)(s+1)}2^{m(s+1)}\Delta_mf \bigg\|_{l^1}\Bigg\|_{L^p}\nonumber\\
		&\lesssim \sum_{j\in\mathbb{Z}}2^{j(s+1)}\chi_{\{j\leq 2\}}\| g\|_{L^\infty}\Big\|\big\|2^{m(s+1)}\Delta_mf \big\|_{l^1}\Big\|_{L^p}\nonumber\\
		&\lesssim\sum_{j\in\mathbb{Z}}2^{j(s+1)}\chi_{\{j\leq 2\}}\| g\|_{L^\infty}\|f\|_{\dot{F}^{s+1}_{p,1}}\\
		&\lesssim\| g\|_{L^\infty}\|f\|_{\dot{F}^{s+1}_{p,1}},
	\end{align*}
	where we used Lemma \ref{Bernstein} and Young's inequality. Regarding to the term $I_5$, thanks to Lemma \ref{guoguji}, we have
	\begin{align*}
		\Big\|\big\|2^{js}|I_5|\big\|_{l^1}\Big\|_{L^p}
		&=\Bigg\|\bigg\|2^{js} \sum_{|m-j|\leq 4}\Delta_j\left(S_{m-1}\partial_ig\Delta_mf^i\right)\bigg\|_{l^1}\Bigg\|_{L^p}\nonumber\\
		&\lesssim\Bigg\|\bigg\|2^{js} \sum_{|m-j|\leq 4}\left[M(|S_{m-1}\partial_ig\Delta_mf^i|^r)(x)\right]^\frac{1}{r}\bigg\|_{l^1}\Bigg\|_{L^p}\nonumber\\
		&=\Bigg\|\bigg\| \sum_{|m-j|\leq 4}2^{(j-m)s}\left[M(|S_{m-1}\partial_ig2^{ms}\Delta_mf^i|^r)(x)\right]^\frac{1}{r}\bigg\|_{l^1}\Bigg\|_{L^p}\nonumber\\
		&\lesssim \| S_{m-1}g\|_{L^\infty}\Bigg\|\bigg\| \sum_{|m-j|\leq 4}2^{(j-m)s}\left[M(|2^{m(s+1)}\Delta_mf^i|^r)(x)\right]^\frac{1}{r}\bigg\|_{l^1}\Bigg\|_{L^p}\nonumber\\
		&\lesssim \| S_{m-1}g\|_{L^\infty}\sum_{j\in\mathbb{Z}}2^{js}\chi_{\{|j|\leq 4\}}\Bigg\|\bigg\| \left[M(|2^{m(s+1)}\Delta_mf^i|^r)(x)\right]^\frac{1}{r}\bigg\|_{l^1}\Bigg\|_{L^p}\nonumber\\
		&\lesssim\sum_{j\in\mathbb{Z}}2^{js}\chi_{\{|j|\leq 4\}}\| g\|_{L^\infty}\|f\|_{\dot{F}^{s+1}_{p,1}}\\
		&\lesssim\| g\|_{L^\infty}\|f\|_{\dot{F}^{s+1}_{p,1}}.
	\end{align*}
	This completed the proof of Proposition \ref{zijixiedejiaohuanzi}.
	\end{proof}
	\begin{remark}\label{s>n-proposition}
		The proof technique of Proposition \ref{zijixiedejiaohuanzi} readily extends to establish further results under the same assumptions and $n\in \mathbb{N}$: for $s>0$, there holds
		\begin{equation*}
			\Big\|\big\|2^{js}([f,\Delta_j]\cdot \nabla g) \big\|_{l^q(j\geq n)}\Big\|_{L^p}\lesssim\sum_{j\geq n}2^{js}\chi_{\{j\leq 4\}}\left(||\nabla f||_{L^\infty}||g||_{\dot{F}^s_{p,q}}+||\nabla g||_{L^\infty}||f||_{\dot{F}^s_{p,q}}\right),
		\end{equation*}
		or for $s>-1$ and $\textup{div}f=0$,
		\begin{equation*}
			\Big\|\big\|2^{js}([f,\Delta_j]\cdot \nabla g) \big\|_{l^q(j\geq n)}\Big\|_{L^p} \lesssim\sum_{j\geq n}2^{js}\chi_{\{j\leq 4\}}\left(||\nabla f||_{L^\infty}||g||_{\dot{F}^s_{p,q}}+|| g||_{L^\infty}||\nabla f||_{\dot{F}^s_{p,q}}\right),
		\end{equation*}
		where $\|f_j\|_{l^q(j\geq n)}$ stands for $\big(\sum_{j\geq n}|f_j|^q\big)^{\frac{1}{q}}$ with usual modification if $q=\infty$.
	\end{remark}

	\section{A priori estimates for transport equation} 
	In this section, based on Proposition \ref{zijixiedejiaohuanzi}, we establish the a priori estimates Theorem \ref{xianyanguji} for the transport equation \eqref{Eq}.
	
	\begin{proof}[Proof of Theorem \ref{xianyanguji}]
		Applying the frequency localization operator $\Delta_j$ to \eqref{Eq}, one has
		\begin{equation}\label{zuoyongDelta}
			\begin{cases}
				(\partial_t+v\cdot\nabla)\Delta_jf=\Delta_j g+[v,\Delta_j]\cdot\nabla f,\\
				\Delta_jf|_{t=0}=\Delta_jf_0.
			\end{cases}
		\end{equation}
		Let us introduce particle trajectory mapping $X(\alpha,t)$, by definition, the solution to the following ordinary differential equation:
		\begin{equation}\label{liziguidaoyinghshe}
			\begin{cases*}
				\partial_tX(t,\alpha)=v(t,X(t,\alpha)),\\
				X(0,\alpha)=\alpha.
			\end{cases*}
		\end{equation}	    	
		Then, it follows from \eqref{zuoyongDelta} that
		\begin{equation}
			\partial_t\left(\Delta_jf\left(t,X(t,\alpha)\right)\right)=\Delta_jg(t,X(t,\alpha))+[v,\Delta_j]\cdot\nabla f(t,X(t,\alpha)).
		\end{equation}
		If we denote the Jacobian determinant of $X(t,\alpha)$ by $J(t,\alpha)=\det(\nabla_\alpha X)(t,\alpha)$, 
		then we have $\partial_tJ(t,\alpha)=J(t,\alpha)(\text{div}v)(t,X(t,\alpha))$. And thus,
		\begin{align*}
			&\mathrel{\phantom{=}}\partial_t\big(J(t,\alpha)\Delta_jf\left(t,X(t,\alpha)\right)\big)\\
			&=J(t,\alpha)\textup{div}v(t,X(t,\alpha))\Delta_jf\left(t,X(t,\alpha)\right)+J(t,\alpha)\Delta_jg(t,X(t,\alpha))\\
			&\mathrel{\phantom{=}}+J(t,\alpha)[v,\Delta_j]\cdot\nabla f(t,X(t,\alpha)),
		\end{align*}
		which together with $J(0,\alpha)=1$ implies that
		\begin{align}
			&\mathrel{\phantom{=}}\big|J(t,\alpha)\Delta_jf\left(t,X(t,\alpha)\right)\big|\nonumber\\
			&\leq |\Delta_jf_0(\alpha)|+\int_{0}^{t}\big|J(\tau,\alpha)\textup{div}v(\tau,X(\tau,\alpha))\Delta_jf\left(\tau,X(\tau,\alpha)\right)\big|d\tau\nonumber\\
			&\mathrel{\phantom{=}}+\int_{0}^{t}\big|J(\tau,\alpha)\Delta_jg(\tau,X(\tau,\alpha))\big|d\tau+\int_{0}^{t}\big|J(\tau,\alpha)[v,\Delta_j]\cdot\nabla f(\tau,X(\tau,\alpha))\big|d\tau\label{jifenhou}.
		\end{align}
		Multiplying $2^{js}$ and taking $l^q$ norm for $j\in\mathbb{Z}$ on both sides of \eqref{jifenhou}, we get by using the Minkowski inequality that
		\begin{align}
			&\mathrel{\phantom{=}}|J(t,\alpha)|\bigg(\sum_{j\in\mathbb{Z} }\left|2^{js}\Delta_jf(t,X(t,\alpha))\right|^{q}\bigg)^{\frac{1}{q}}\nonumber\\
			&\leq\bigg(\sum_{j\in\mathbb{Z} }\left|2^{js}\Delta_jf_0(\alpha)\right|^{q}\bigg)^{\frac{1}{q}}\nonumber+\int_{0}^{t}|J(\tau,\alpha)|(\textup{div}v)(\tau,X(\tau,\alpha))\bigg(\sum_{j\in\mathbb{Z} }\left|2^{js}\Delta_jf(\tau,X(\tau,\alpha))\right|^{q}\bigg)^{\frac{1}{q}}d\tau\nonumber\\
			&\mathrel{\phantom{=}}+\int_{0}^{t}|J(\tau,\alpha)|\bigg(\sum_{j\in\mathbb{Z} }\left|2^{js}\Delta_jg(\tau,X(\tau,\alpha))\right|^{q}\bigg)^{\frac{1}{q}}d\tau\nonumber\\
			&\mathrel{\phantom{=}}+\int_{0}^{t}|J(\tau,\alpha)|\bigg(\sum_{j\in\mathbb{Z} }\left|2^{js}\Big(\big([v,\Delta_j]\cdot\nabla f\big)(\tau,X(\tau,\alpha))\Big)\right|^{q}\bigg)^{\frac{1}{q}}d\tau\label{lqhou},
		\end{align}		
		with the  usual modification if $q=\infty$. Next, taking the $L^p$ norm with respect to $\alpha\in\mathbb{R}^d$ on both sides of \eqref{lqhou}, we get by using the Minkowski inequality that
		\begin{align}
			&\mathrel{\phantom{=}}\Bigg(\int_{\mathbb{R}^d}\Big|J(\tau,\alpha)\Big(\sum_{j\in\mathbb{Z} }\Big|2^{js}\Delta_jf(t,X(t,\alpha))\Big|^{q}\Big)^{\frac{1}{q}}\Big|^pd\alpha\Bigg)^\frac{1}{p}\nonumber\\
			&\leq\|f_0\|_{\dot{F}^s_{p,q}}+\int_{0}^{t}\Bigg(\int_{\mathbb{R}^d}\Big|J(\tau,\alpha)(\textup{div}v)(\tau,X(\tau,\alpha))\Big(\sum_{j\in\mathbb{Z} }\left|2^{js}\Delta_jf(\tau,X(\tau,\alpha))\right|^{q}\Big)^{\frac{1}{q}}\Big|^pd\alpha\Bigg)^\frac{1}{p}d\tau\nonumber\\
			&+\int_{0}^{t}\Bigg(\int_{\mathbb{R}^d}\Big|J(\tau,\alpha)\Big(\sum_{j\in\mathbb{Z} }\big|2^{js}\Delta_jg(\tau,X(\tau,\alpha))\big|^{q}\Big)^{\frac{1}{q}}\Big|^pd\alpha\Bigg)^\frac{1}{p}d\tau\nonumber\\
			&+\int_{0}^{t}\Bigg(\int_{\mathbb{R}^d}\Big|J(\tau,\alpha)\Big(\sum_{j\in\mathbb{Z} }\big|2^{js}\big(([v,\Delta_j]\cdot\nabla f)(\tau,X(\tau,\alpha))\big)\big|^{q}\Big)^{\frac{1}{q}}\Big|^pd\alpha\Bigg)^\frac{1}{p}d\tau,\label{Lqhou}
		\end{align}	
		with the  usual modification if $p=\infty$. 
		Thus, it follows from \eqref{Lqhou} and the change of variables formula that
		\begin{align}
			\|f\|_{\dot{F}^s_{p,q}}&\leq\|f_0\|_{\dot{F}^s_{p,q}}+\int_{0}^{t}\|\textup{div}v\|_{L^\infty}||f||_{\dot{F}^s_{p,q}}d\tau+\int_{0}^{t}\|g\|_{\dot{F}^s_{p,q}}d\tau\nonumber\\
			&\quad+\int_{0}^{t}\Big\|\big\|2^{js}[v,\Delta_j]\cdot\nabla f\big\|_{l^q}\Big\|_{L^p}d\tau\label{cedububian}.
		\end{align}
		Thanks to \eqref{qudiaodiv=0} in Proposition \ref{zijixiedejiaohuanzi}, the last term on the right of \eqref{cedububian} is dominated by 
		\begin{equation}\label{jiaohuanzikongzhi}
			\int_{0}^{t}\|\nabla v\|_{L^\infty}\|f\|_{\dot{F}^s_{p,q}}+\|\nabla f\|_{L^\infty}\|v\|_{\dot{F}^s_{p,q}}d\tau,
		\end{equation} 
		or when $\text{div}v=0$ and $s>-1$, it is controlled by
		\begin{equation}\label{jiaohuanzikongzhi-divv=0}
			\int_{0}^{t}\|\nabla v\|_{L^\infty}\|f\|_{\dot{F}^s_{p,q}}+\| f\|_{L^\infty}\|\nabla v\|_{\dot{F}^s_{p,q}}d\tau.
		\end{equation}
		Substituting \eqref{jiaohuanzikongzhi} and \eqref{jiaohuanzikongzhi-divv=0} into \eqref{cedububian}, respectively, one yields for $s>0$,
		\begin{align}
			\|f\|_{\dot{F}_{p,q}^s} \leq &\|f_0\|_{\dot{F}_{p,q}^s} + \int_0^t \|{\rm div}v\|_{L^\infty} \|f\|_{\dot{F}_{p,q}^s}d\tau + \int_0^t \|g\|_{\dot{F}_{p,q}^s} d\tau \nonumber \\
			&+ C\int_0^t \big( \|\nabla v\|_{L^\infty} \|f\|_{\dot{F}_{p,q}^s} + \|\nabla f\|_{L^\infty} \|v\|_{\dot{F}_{p,q}^s} \big) d\tau,\label{qiciF-priori-estimates}
		\end{align}
		and for $s>-1$ and $\text{div }v=0$,
		\begin{align}
			\|f\|_{\dot{F}_{p,q}^s} \leq &\|f_0\|_{\dot{F}_{p,q}^s} + \int_0^t \|{\rm div}v\|_{L^\infty} \|f\|_{\dot{F}_{p,q}^s}d\tau + \int_0^t \|g\|_{\dot{F}_{p,q}^s} d\tau \nonumber \\
			&+ C\int_0^t \big( \|\nabla v\|_{L^\infty} \|f\|_{\dot{F}_{p,q}^s} + \| f\|_{L^\infty} \|\nabla v\|_{\dot{F}_{p,q}^s} \big) d\tau.\label{qiciF-priori-estimates-divv=0}
		\end{align}
		If $s>1+\frac{d}{p}$ when $p>1$ or $s\geq 1+d$ when $p=1$, and $\nabla v \in L^1(0,T;F^{s-1}_{p,q}(\mathbb{R}^d))$, we obtain
		\begin{align}
			\|f\|_{\dot{F}^s_{p,q}}&\leq\|f_0\|_{\dot{F}^s_{p,q}}+\int_{0}^{t}\|\textup{div}v\|_{L^\infty}||f||_{\dot{F}^s_{p,q}}d\tau+\int_{0}^{t}\|g\|_{\dot{F}^s_{p,q}}d\tau+C\int_{0}^{t}\|\nabla v\|_{F^{s-1}_{p,q}}\| f\|_{F^{s}_{p,q}}d\tau\nonumber\\
			&\leq\|f_0\|_{\dot{F}^s_{p,q}}+\int_{0}^{t}\|g\|_{\dot{F}^s_{p,q}}d\tau+C\int_{0}^{t}\|\nabla v\|_{F^{s-1}_{p,q}}\| f\|_{F^{s}_{p,q}}d\tau\label{qiciFguii},
		\end{align}
		here we used Lemma \ref{lem:Triebel-Lizorkin-properties} $(ii)$, Lemma \ref{tiduheqiciF} and the fact  $\|f\|_{F^s_{p,q}}=\|f\|_{\dot{F}^s_{p,q}}+\|f\|_{L^p}$ as $s>0$.
		
		Now, we estimate the $L^p$ norm of $f$. Multiplying $\text{sgn}(f)|f|^{p-1}$ on both sides of \eqref{Eq}, and integrating the resulting equation over $\mathbb{R}^d$, we deduce
		\begin{align}
			\|f\|_{L^p}&\leq \|f_0\|_{L^p}+\int_{0}^{t}\|g\|_{L^p}d\tau+\frac{1}{p}\int_{0}^{t}\|\text{div} v\|_{L^\infty}\|f\|_{L^p}d\tau\label{equation13}\\
			&\leq \|f_0\|_{L^p}+\int_{0}^{t}\|g\|_{L^p}d\tau+C\int_{0}^{t}\|\nabla v\|_{F^{s-1}_{p,q}}\|f\|_{F^s_{p,q}}d\tau,\label{equation14}
		\end{align}
		Hence, from \eqref{equation13} together with \eqref{qiciF-priori-estimates} and \eqref{qiciF-priori-estimates-divv=0}  we obtain \eqref{feiqiciF-priori-estimates} and \eqref{feiqiciF-priori-estimates-divv=0}, respectively, and from substituting \eqref{equation14} into $\eqref{qiciFguii}$, we arrive at \eqref{xianyangujijifenshizi}. Then applying the Gronwall inequality, one reaches \eqref{xianyangujishizi}. If $f=v$, $v\in L^1(0,T;L^\infty)$, $s>0$ , we find that \eqref{jiaohuanzikongzhi} reduces to $\int_{0}^t\|\nabla v\|_{L^\infty}\|f\|_{\dot{F}^s_{p,q}}d\tau$. A slight modification of the preceding proof yields the desired result for this case. Therefore, we complete the proof of Theorem \ref{xianyanguji}.
	\end{proof} 
	   
	\section{Local well-posedness for transport equation}
	We now address the local well-posedness result Theorem \ref{shuyunfangchengcunzaixing} for the transport equation \eqref{Eq} with data in the Triebel-Lizorkin spaces. 
	
	\begin{proof}[Proof of Theorem \ref{shuyunfangchengcunzaixing}]
		We first smooth out the data and the velocity filed $v$ by setting 
		\begin{align*}
			f^n_0\triangleq S_nf_0,\qquad g^n=\rho _n\ast_tS_ng\quad \text{and} \quad v^n=\rho _n\ast_tS_nv,
		\end{align*}
		where $\rho_n\triangleq\rho_n(t)$ stands for a sequence of mollifiers with respect to the time variable. We clearly have $f^n_0\in F^\infty_{p,q},g^n\in C([0,T];F^\infty_{p,q}),v^n\in C([0,T]\times\mathbb{R}^d)$ and $\nabla v^n\in C([0,T];F^\infty_{p,q})$ with $F^\infty_{p,q}\triangleq\cap_{s\in\mathbb{R}}F^s_{p,q}$. Moreover,  $f^n_0$ is uniformly bounded in $F^s_{p,q}$, $g^n$ is uniformly bounded in $L^1(0,T;F^\infty_{p,q})$, $v^n$ is uniformly bounded in  $L^\rho(0,T;F^{-M}_{\infty,\infty})$ and $\nabla v^n$ is uniformly  bounded in  $L^1(0,T;F^{s-1}_{p,q})$.
		
		Let $f^n$ be the solution to the following equation:
		\begin{equation}\label{nEq}
			\begin{cases*}
				\partial_tf^n+v^n\cdot\nabla f^n=g^n,\\
				f^n|_{t=0}=f^n_0.
			\end{cases*}		
		\end{equation}
		Clearly, $f^n$ is smooth, and according to Theorem \ref{xianyanguji}, one has 
		\begin{equation}
			\|f^n\|_{F^s_{p,q}}\leq e^{C\int_{0}^{t}\|\nabla v^n(\tau)\|_{F^{s-1}_{p,q}}d\tau}\bigg(\|f^n_0\|_{F^s_{p,q}}+\int_{0}^{t}\|g^n(\tau)\|_{F^s_{p,q}}e^{-C\int^\tau_0\|\nabla v^n(\tau')\|_{F^{s-1}_{p,q}}d\tau'}\bigg).
		\end{equation}
		Thus, in view of the uniform bounds for $f^n_0,g^n$ and $v^n$, we conclude that the sequence $\{f^n\}_{n\in\mathbb{N}}$ is uniformly bounded in $C([0,T];F^s_{p,q})$.
		
		In order to prove the convergence of a subsequence, we appeal to compactness arguments. Firstly, notice that
		\begin{equation}\label{equation12}
			\partial_tf^n-g^n=-v^n\cdot \nabla f^n.
		\end{equation}
		Since $\nabla f^n$ is uniformly bounded in $L^\infty(0,T;F^{s-1}_{p,q})$ and $v^n$ is uniformly bounded in $L^1(0,T;F^{s}_{p,q})$, one can conclude by appealing to Lemma \ref{lem:product} that the right hand-side of \eqref{equation12} is uniformly bounded in $L^1(0,T;F^{s-1}_{p,q})$.  Integrating in time and denoting $\bar{f}^n(t)\triangleq f^n(t)-\int_{0}^tg^n(\tau)d\tau$,  we thus gather that there exists some $\beta>0$ such that the sequence  $\{\bar{f}^n\}_{n\in\mathbb{N}}$ is uniformly bounded in $C^\beta([0,T];F^{s-1}_{p,q})$, hence uniformly equicontinuous with values in $F^{s-1}_{p,q}$.
		
		Next, observe that the map $f\mapsto\phi f$ is compact from $F^s_{p,q}$ to $ F^{s-1}_{p,q}$ for all $\phi\in C^\infty_c$ (by virtue of the embedding property on page 60 in \cite{MR1410258} and Theorem 5.1.3 in \cite{MR1163193}). Combining the Arzel\`{a}-Ascoli theorem and the Cantor diagonal process thus ensures that, up to a subsequence, the sequence $\{\bar{f}^n\}_{n\in\mathbb{N}}$ converges in $\mathscr{S}'$ to some distribution $\bar{f}$ such that $\phi\bar{f}\in C([0,T];F^{s-1}_{p,q})$ for all $\phi\in C^\infty_c$. 
		
		Finally, appealing once again to the uniform bounds in $L^\infty([0,T];F^s_{p,q})$ and the Fatou property (Lemma \ref{lem:Triebel-Lizorkin-properties} ($iii$)) for Triebel-Lizorkin spaces, we get $\bar{f}\in L^\infty([0,T];F^{s}_{p,q})$. By an interpolation argument, together with the bounds in $L^\infty([0,T];F^{s}_{p,q})$ for $\{\bar{f}^n\}_{n\in\mathbb{N}}$, we find that $\phi\bar{f}^n\mapsto \phi\bar{f}$ in $C([0,T];F^{s'}_{p,q})$ for any $s'<s$ and $\phi\in C^\infty_c$ so that we may pass to the limit in the equation for $f^n$, in the sense of distribution. Besides, the sequences $\{f_0^n\}_{n\in\mathbb{N}}$, $\{g^n\}_{n\in\mathbb{N}}$ and  $\{v^n\}_{n\in\mathbb{N}}$ converge respectively to $f_0$, $g$, and $v$, which may be easily deduced from their definitions. We conclude that the function $f\triangleq\bar{f}+\int_{0}^tg(\tau)d\tau$ is a solution to \eqref{Eq}.

		We still have to prove that $f\in C([0,T];F^s_{p,q})$ in the case where $q< \infty$. Just by looking at the equation \eqref{Eq}, it is easy to get $\partial_tf\in L^1(0,T;F^{-M'}_{p,\infty})$ for some large enough $M'$. Hence $f\in C([0,T];F^{-M'}_{p,\infty})$, whence $S_nf\in C([0,T];F^s_{p,q})$ for all $n\in\mathbb{N}$. Note that 
		\begin{equation*}
			\Delta_j(f-S_nf)=\sum_{\substack{|j-j'|\leq 1\\j'\geq n}}\Delta_j\Delta_{j'}f.
		\end{equation*}
		For $1<p,q<\infty$, using Lemma \ref{miaoguji}, Young's inequality and Lemma \ref{Fguji}, we have 
		\begin{align}
			\|f-S_nf\|_{\dot{F}^s_{p,q}}&=\Big\|\big\|\sum_{\substack{|j-j'|\leq 1\\j'\geq n}}2^{js}\Delta_j\Delta_{j'}f\big\|_{l^q}\Big\|_{L^p}\nonumber\\
			&\leq C\Big\|\big\|\sum_{\substack{|j-j'|\leq 1\\j'\geq n}}2^{(j-j')s}2^{j's}M(\Delta_{j'}f)\big\|_{l^q}\Big\|_{L^p}\nonumber\\
			&\leq  C\Big\|\big\|2^{j's}\Delta_{j'}f\big\|_{l^q(j'\geq n)}\Big\|_{L^p}.\label{equation19}		
		\end{align}
		The above \eqref{equation19} also holds true for $p = 1$ or $q=1$, since Lemma \ref{guoguji}, Young's inequality and Lemma \ref{Fguji}.
		Similar to the proof of \eqref{qiciF-priori-estimates}, for any $n\in\mathbb{N}$, consider  $j'\geq n$ instead of $j'\in\mathbb{Z}$, we can show that
		\begin{align*}
			\Big\|\big\|2^{j's}\Delta_{j'}f\big\|_{l^q(j'\geq n)}\Big\|_{L^p}&\leq \Big\|\big\|2^{j's}\Delta_{j'}f_0\big\|_{l^q(j'\geq n)}\Big\|_{L^p}+\int_{0}^t\|\text{div} v\|_{L^\infty}\Big\|\big\|2^{j's}\Delta_{j'}f\big\|_{l^q(j'\geq n)}\Big\|_{L^p}d\tau\nonumber\\
			&\quad+\int_{0}^t\Big\|\big\|2^{j's}\Delta_{j'}g\big\|_{l^q(j'\geq n)}\Big\|_{L^p}+\int_{0}^{t}\Big\|\big\|2^{j's}[v,\Delta_j]\cdot\nabla f\big\|_{l^q(j'\geq n)}\Big\|_{L^p}d\tau,
		\end{align*}
		which along with \eqref{equation19} and Gronwall's inequality yields
		\begin{align}
			\|f-S_nf\|_{L^\infty_T(\dot{F}^s_{p,q})}&\leq Ce^{C\int_{0}^{T}\|\nabla v\|_{L^\infty}d\tau}\Big(\Big\|\big\|2^{j's}\Delta_{j'}f_0\big\|_{l^q(j'\geq n)}\Big\|_{L^p}+\int^T_0\Big\|\big\|2^{j's}\Delta_{j'}g\big\|_{l^q(j'\geq n)}\Big\|_{L^p}d\tau\nonumber\\
			&\quad+\int_{0}^T\Big\|\big\|2^{j's}[v,\Delta_j]\cdot\nabla f\big\|_{l^q(j'\geq n)}\Big\|_{L^p}d\tau\Big)\label{equation3}.
		\end{align}
		The first term of right hand-side in \eqref{equation3} clearly tends to zero when $n$ goes to infinity. Thanks to Remark \ref{s>n-proposition}, the commutator term in the third term of right hand-side of \eqref{equation3} tends to zero when $n$ goes to infinity. So, the terms in the integrals approach to zero for almost every $t$. Hence, by virtue of Lebesgue's dominated convergence theorem, $\|f-S_nf\|_{L^\infty_T(\dot{F}^s_{p,q})}$ tends to zero when $n$ goes to infinity. Moreover, $\|f-S_nf\|_{L^p}\to 0\ (n\to\infty)$ for all $t\in[0,T]$. Thus, we conclude that  $\left\|f-S_nf\right\|_{L^\infty_T(F^s_{p,q})}\rightarrow 0 $ when $n\rightarrow\infty$. This achieves to proving that $f\in C([0,T];F^s_{p,q})$ in the case $q<\infty$.
		
		When $q=\infty$, note that for any $s'<s$, we have the embedding $F^s_{p,\infty}\hookrightarrow F^{s'}_{p,1}$ so that the above argument may be repeated in the space $F^{s'}_{p,1}$, this yields  $f\in C([0,T];F^{s'}_{p,1})$.
		
		For the  uniqueness and continuity with respect to the initial data, if we are given  $(f^1,f^2)\in L^\infty(0,T;F^s_{p,q}\times F^s_{p,q})\cap C([0,T];\mathscr{S}'\times\mathscr{S}')$ two solutions to \eqref{Eq} with initial data $f^1_0,f^2_0\in F^s_{p,q}$. Denote $w=f^1-f^2$, then $w\in L^\infty(0,T;F^s_{p,q})$ solves the following  transport equation:
		\begin{equation}
			\begin{cases}
				\partial_tw+v\cdot\nabla w=0,\\
				w|_{t=0}=f^1_0-f^2_0.
			\end{cases}			
		\end{equation} 
		In view of \eqref{xianyangujishizi} in Theorem \ref{xianyanguji}, we have for every $t\in[0,T],$
		\begin{equation}
			||f^1-f^2||_{F^{s}_{p,q}}\leq e^{C\int_{0}^t\|\nabla v\|_{F^{s-1}_{p,q}}d\tau}\|f^1_0-f^2_0\|_{F^{s}_{p,q}},
		\end{equation}
		which implies the uniqueness of the solution and its continuous dependence.
		Therefore, we complete the proof of Theorem \ref{shuyunfangchengcunzaixing}. 
	\end{proof}
	
	\section{Solving the ideal MHD equations in Triebel-Lizorkin spaces}
	In this section, we apply the preceding theory for transport equation \eqref{Eq} to establish the local well-posedness and blow-up criteria for the ideal MHD equations by proving Theorems \ref{MHD-Local-well-posedness} and \ref{Blow-up Criterion}.
	
	\begin{proof}[Proof of Theorem \ref{MHD-Local-well-posedness}]
		\textbf{Existence:} We prove it in the following four steps.\\
		\textbf{Step 1. Approximate solutions.}
		
		Staring from $(u^{(0)},b^{(0)})=(0,0)$, from Theorem \ref{shuyunfangchengcunzaixing} we then define by induction a sequence of smooth functions $(u^{(n)},b^{(n)})_{n\in\mathbb{N}}$ by solving the following systems:
		\begin{equation}\label{jinsiMHD}
			\begin{cases}
				\partial_tu^{(n+1)}+u^{(n)}\cdot\nabla u^{(n+1)}-b^{(n)}\cdot\nabla b^{(n+1)}=-\nabla\tilde{\pi}_1^{(n+1)}\\
				\partial_tb^{(n+1)}+u^{(n)}\cdot\nabla b^{(n+1)}-b^{(n)}\cdot\nabla u^{(n+1)}=-\nabla\tilde{\pi}_2^{(n+1)}\\
				\nabla\cdot b^{(n+1)}=\nabla \cdot u^{(n+1)}=0,\\
				(u^{(n+1)},b^{(n+1)})|_{t=0}=S_{n+2}(u_0,b_0),
			\end{cases}
		\end{equation}
		where $\tilde{\pi}_1^{(n+1)}=p^{(n+1)}+\frac{1}{2}(b^{(n+1)})^2$ and $\tilde{\pi}_2^{(n+1)}=0$.
		
		We set 
		\begin{equation*}
			z^{+^{(n)}}=u^{(n)}+b^{(n)},\quad z^{-^{(n)}}=u^{(n)}-b^{(n)}.
		\end{equation*}
		Then \eqref{jinsiMHD} can be reduced to
		\begin{equation}\label{jinsiIMHD}
			\begin{cases}
				\partial_tz^{+^{(n+1)}}+(z^{-^{(n)}}\cdot\nabla)z^{+{(n+1)}}=-\nabla\pi_1^{(n+1)},\\
				\partial_tz^{-^{(n+1)}}+(z^{+^{(n)}}\cdot\nabla)z^{-{(n+1)}}=-\nabla\pi_2^{(n+1)},\\
				\nabla\cdot z^{+^{(n+1)}}=\nabla\cdot z^{-^{(n+1)}}=0,\quad\forall n\in\mathbb{N}\\
				z^{+^{(n+1)}}(0)=S_{n+2}z^+_0,\quad z^{-^{(n+1)}}(0)=S_{n+2}z^-_0,
			\end{cases}
		\end{equation}
		where $\pi_1^{(n+1)}=\pi_2^{(n+1)}=p^{(n+1)}+\frac{1}{2}(b^{(n+1)})^2 $ and $(z^{+^{(0)}},z^{-^{(0)}})=(0,0)$.\\
		\textbf{Step 2. Uniform bounds.}
		
		According to \eqref{xianyangujijifenshizi} in Theorem \ref{xianyanguji}, we have the following inequality for all $n\in\mathbb{N}$,
		\begin{align}		
			&\mathrel{\phantom{=}}\|z^{+^{(n+1)}}\|_{F^s_{p,q}}+\|z^{-^{(n+1)}}\|_{F^s_{p,q}}\nonumber\\
			&\leq C(\|z^{+}_0\|_{F^s_{p,q}}+\|z^{-}_0\|_{F^s_{p,q}})+\int_0^t\|\nabla\pi^{(n+1)}_1\|_{F^s_{p,q}}+\|\nabla\pi^{(n+1)}_2\|_{F^s_{p,q}}d\tau\nonumber\\
			&\mathrel{\phantom{=}}+C\int^t_0\|\nabla z^{-^{(n)}}\|_{F^{s-1}_{p,q}}\|z^{+^{(n+1)}}\|_{F^s_{p,q}}+\|\nabla z^{+^{(n)}}\|_{F^{s-1}_{p,q}}\|z^{-^{(n+1)}}\|_{F^s_{p,q}}d\tau,\label{jinsiIMHDxianyangujihou}	
		\end{align}
		where we used the fact that 
		\begin{equation*}
			\|S_{n+2}z^{+}_0\|_{F^s_{p,q}}+\|S_{n+2}z^{-}_0\|_{F^s_{p,q}}\leq C(\|z^{+}_0\|_{F^s_{p,q}}+\|z^{-}_0\|_{F^s_{p,q}})
		\end{equation*}
		for some constant $C$ independent of $n$, which is ensured by Lemma \ref{lem:Snguji}. Taking the divergence on both sides of the first equation in \eqref{jinsiIMHD}, we obtain the following representation of the pressure:
		\begin{equation}\label{pi_1Riesz}
			\pi_1^{(n+1)}=(-\Delta)^{-1}(\partial_jz_i^{-^{(n)}}\partial_iz_j^{+^{(n+1)}})=(-\Delta)^{-1}\partial_j\partial_i(z_i^{-^{(n)}}z_j^{+^{(n+1)}}).
		\end{equation}
		For $l,m\in[1,d]$, we have
		\begin{equation*}
			\partial_l\partial_m\pi^{(n+1)}_1=(-\Delta)^{-1}\partial_l\partial_m(\partial_jz_i^{-^{(n)}}\partial_iz_j^{+^{(n+1)}})=\mathscr{R}_{l}\mathscr{R}_{m}(\partial_jz_i^{-^{(n)}}\partial_iz_j^{+^{(n+1)}}),
		\end{equation*}
		where $\mathscr{R}_l$ denotes the Riesz transform. Thanks to the boundedness of the Riesz transform in the homogeneous Triebel-Lizorkin spaces \cite{MR1009119}, Proposition \ref{moser-type-estimate} and Lemma \ref{tiduheqiciF}, one gets
		\begin{align}
			\|\nabla\pi^{(n+1)}_1\|_{\dot{F}^s_{p,q}}
			&\leq C\|\partial_jz_i^{-^{(n)}}\partial_iz_j^{+^{(n+1)}}\|_{\dot{F}^{s-1}_{p,q}}\nonumber\\
			&\leq C\|\nabla z^{-^{(n)}}\|_{L^\infty}\|\nabla z^{+^{(n+1)}}\|_{\dot{F}^{s-1}_{p,q}}+\|\nabla z^{+^{(n+1)}}\|_{L^\infty}\|\nabla z^{-^{(n)}}\|_{\dot{F}^{s-1}_{p,q}}\nonumber\\
			&\leq C\|\nabla z^{-^{(n)}}\|_{L^\infty}\| z^{+^{(n+1)}}\|_{\dot{F}^{s}_{p,q}}+\|\nabla z^{+^{(n+1)}}\|_{L^\infty}\| z^{-^{(n)}}\|_{\dot{F}^{s}_{p,q}}.
		\end{align}	
		Using Young's inequality and Lemma \ref{p=1Reszi}, we get
		\begin{align*}
			\|S_0\nabla\pi^{(n+1)}_1\|_{L^p}&=C\|S_0\partial_k(-\Delta)^{-1}\partial_j\partial_i(z^{-^{(n)}}_iz^{+^{(n+1)}}_j)\|_{L^p}\\
			&=C\big\|\mathscr{F}^{-1}\big(m(\xi)\xi_k\big)\ast\big(z^{-^{(n})}_i z^{+^{(n+1)}}_j\big)\big\|_{L^p}\\
			&\leq C\| z^{+^{(n+1)}}\|_{L^\infty}\| z^{-^{(n)}}\|_{L^p}.
		\end{align*}
		Now, from \cite{MR1163193}, $\|S_0f\|_{L^p}+\|f\|_{\dot{F}^s_{p,q}}$  is an equivalent norm in $F^s_{p,q}$ for $s>0,(p,q)\in[1,\infty)\times[1,\infty]$ or $p=q=\infty$, we have 
		\begin{equation}\label{pi1}
			\|\nabla\pi^{(n+1)}_1\|_{F^s_{p,q}}\leq C\|z^{-^{(n)}}\|_{F^s_{p,q}}\|z^{+^{(n+1)}}\|_{F^s_{p,q}},
		\end{equation}
		where we used $F^{s-1}_{p,q}(\mathbb{R}^d)\hookrightarrow  L^\infty(\mathbb{R}^d)$. 
		Similar to the proof of \eqref{pi1}, we conclude that 
		\begin{equation}\label{pi2}
			\|\nabla\pi^{(n+1)}_2\|_{F^s_{p,q}}\leq C\|z^{+^{(n)}}\|_{F^s_{p,q}}\|z^{-^{(n+1)}}\|_{F^s_{p,q}}.
		\end{equation}
		By summing up \eqref{jinsiIMHDxianyangujihou}, \eqref{pi1} and \eqref{pi2}, we obtain
		\begin{align*}		
			&\mathrel{\phantom{=}}\|z^{+^{(n+1)}}\|_{F^s_{p,q}}+\|z^{-^{(n+1)}}\|_{F^s_{p,q}}\\
			&\leq C(\|z^{+}_0\|_{F^s_{p,q}}+\|z^{-}_0\|_{F^s_{p,q}})+C\int^t_0\bigg(\| z^{-^{(n)}}\|_{F^{s}_{p,q}}\|z^{+^{(n+1)}}\|_{F^s_{p,q}}+\| z^{+^{(n)}}\|_{F^{s}_{p,q}}\|z^{-^{(n+1)}}\|_{F^s_{p,q}}\bigg)d\tau\\
			&\leq C(\|z^{+}_0\|_{F^s_{p,q}}+\|z^{-}_0\|_{F^s_{p,q}})+C\int^t_0\big(\| z^{-^{(n)}}\|_{F^{s}_{p,q}}+\| z^{+^{(n)}}\|_{F^{s}_{p,q}}\big)\big(\|z^{-^{(n+1)}}\|_{F^s_{p,q}}+\|z^{+^{(n+1)}}\|_{F^s_{p,q}}\big)d\tau.
		\end{align*}
		Applying Gronwall's inequality, one gets
		\begin{equation}
			\|z^{+^{(n+1)}}\|_{F^s_{p,q}}+\|z^{-^{(n+1)}}\|_{F^s_{p,q}}\leq C(\|z^{+}_0\|_{F^s_{p,q}}+\|z^{-}_0\|_{F^s_{p,q}})e^{C\int_0^t(\|z^{+^{(n)}}\|_{F^s_{p,q}}+\|z^{-^{(n)}}\|_{F^s_{p,q}})d\tau},
		\end{equation}
		which ensures that there exists $0<T_0<\frac{\ln2}{2C(\|z^{+}_0\|_{F^s_{p,q}}+\|z^{-}_0\|_{F^s_{p,q}})}$ such that for all $n,t\in[0,T_0]$,
		\begin{equation}\label{IMHDyoujie}
			\|z^{+^{(n)}}\|_{F^s_{p,q}}+\|z^{-^{(n)}}\|_{F^s_{p,q}}\leq 2C(\|z^+_0\|_{F^s_{p,q}}+\|z^-_0\|_{F^s_{p,q}}).
		\end{equation}
		\textbf{Step 3. Convergence.}
		
		We now proceed to prove that there exists a time $T\in(0,T_0]$ independent of $n$ such that $(z^{+^{(n)}},z^{-^{(n)}})_{n\in\mathbb{N}}$ is a Cauchy sequence in $C([0,T];F^{s-1}_{p,q}\times F^{s-1}_{p,q})$. For this purpose, we set
		\[ \delta z^{+^{(n+1)}}=z^{+^{(n+1)}}-z^{+^{(n)}},\quad \delta z^{-^{(n+1)}}=z^{-^{(n+1)}}-z^{-^{(n)}}, \]
		\[ \delta \pi^{(n+1)}_j=\pi^{(n+1)}_j-\pi^{(n)}_j,\quad j=1,2. \]
		From \eqref{jinsiIMHD} for all $n\in\mathbb{N}$, we have
		\begin{equation}\label{IMHDCauchy}
			\begin{cases}
				\partial_t\delta z^{+^{(n+1)}}+z^{-^{(n)}}\cdot\nabla\delta z^{+{(n+1)}}=-\delta z^{-^{(n)}}\cdot\nabla z^{+{(n)}}-\nabla\delta\pi_1^{(n+1)},\\
				\partial_t\delta z^{-^{(n+1)}}+z^{+^{(n)}}\cdot\nabla\delta z^{-{(n+1)}}=-\delta z^{+^{(n)}}\cdot\nabla z^{-{(n)}}-\nabla\delta\pi_2^{(n+1)},\\
				(\delta z^{+^{(n+1)}},\delta z^{-^{(n+1)}})|_{t=0}=\Delta_{n+1}(z^+_0,z^-_0).
			\end{cases}
		\end{equation}
		By means of \eqref{feiqiciF-priori-estimates-divv=0} in Theorem \ref{xianyanguji} and the embedding $F^{s-1}_{p,q}\hookrightarrow L^\infty$, one infers
		\begin{align}		
			&\mathrel{\phantom{=}}\|\delta z^{+^{(n+1)}}\|_{F^{s-1}_{p,q}}+\|\delta z^{-^{(n+1)}}\|_{F^{s-1}_{p,q}}\nonumber\\
			&\leq\|\Delta_{n+1}z^{+}_0\|_{F^{s-1}_{p,q}}+\|\Delta_{n+1}z^{-}_0\|_{F^{s-1}_{p,q}}+\int_0^t\|\nabla\delta\pi^{(n+1)}_1\|_{F^{s-1}_{p,q}}+\|\nabla\delta\pi^{(n+1)}_2\|_{F^{s-1}_{p,q}}d\tau\nonumber\\
			&\mathrel{\phantom{=}}+\int^t_0\|\delta z^{-^{(n)}}\cdot\nabla z^{+^{(n)}}\|_{F^{s-1}_{p,q}}+\|\delta z^{+^{(n)}}\cdot\nabla z^{-^{(n)}}\|_{F^{s-1}_{p,q}}d\tau\nonumber\\
			&\mathrel{\phantom{=}}+C\int^t_0\| z^{-^{(n)}}\|_{F^{s}_{p,q}}\|\delta z^{+^{(n+1)}}\|_{F^{s-1}_{p,q}}+\| z^{+^{(n)}}\|_{F^{s}_{p,q}}\|\delta z^{-^{(n+1)}}\|_{F^{s-1}_{p,q}}d\tau.\label{equation5}
		\end{align}
		Thanks to \eqref{a.o.c}, Lemma \ref{miaoguji}, Young's inequality and  Lemma \ref{Fguji}, we have for $1<p<\infty$, $1<q\leq\infty$,
		\begin{align*}
			\|\Delta_{n+1}z^+_0\|_{\dot{F}^{s-1}_{p,q}}&=\Big\| \big\|2^{j(s-1)}|\Delta_j\Delta_{n+1}z^+_0|\big\|_{l^q}\Big\|_{L^p}\\
			&\leq \Big\| \big\|\sum_{ |j-n-1|\leq 1}2^{s(j-n-1)}2^{-j}M(2^{{(n+1)s}}|\Delta_{n+1}z^+_0|)\big\|_{l^q}\Big\|_{L^p}\\
			&\leq C2^{-n}\|z^+_0\|_{\dot{F}^s_{p,q}}.
		\end{align*}
		The endpoint cases where $p=1$ or $q=1$ with $1\leq p<\infty$ can be handled by replacing Lemma \ref{miaoguji} with Lemma \ref{guoguji} in the above argument. Meanwhile, using Remark \ref{remark1}, we get
		\begin{align*}
			\|\Delta_{n+1}z^+_0\|_{\dot{F}^{s-1}_{\infty,\infty}}=\sup_{j\in \mathbb{Z}}2^{j(s-1)}\|\Delta_{n+1}\Delta_jz^+_0\|_{L^\infty}
			\leq C2^{-n}\sup_{j\in \mathbb{Z}}2^{js}\|\Delta_jz^+_0\|_{L^\infty}
			=C2^{-n}\|z^+_0\|_{\dot{F}^{s}_{\infty,\infty}}.
		\end{align*}
			%\begin{align*}
			%\|\Delta_{n+1}z^+_0\|_{\dot{F}^{s-1}_{\infty,\infty}}&=\sup_{j\in \mathbb{Z}}2^{j(s-1)}\|\Delta_{n+1}\Delta_jz^+_0\|_{L^\infty}\\
			%&\leq C2^{-n}\sup_{j\in \mathbb{Z}}2^{js}\|\Delta_jz^+_0\|_{L^\infty}\\
			%&=C2^{-n}\|z^+_0\|_{\dot{F}^{s}_{\infty,\infty}}.
		%\end{align*}
		Similarly, an estimate for $\|\Delta_{n+1}z^-_0\|_{\dot{F}^{s-1}_{p,q}} $ can be obtained in the same manner as above, thereby yields
		\begin{equation}
			\|\Delta_{n+1}z^+_0\|_{\dot{F}^{s-1}_{p,q}}+\|\Delta_{n+1}z^-_0\|_{\dot{F}^{s-1}_{p,q}}\leq C2^{-n}\left( \|z^+_0\|_{\dot{F}^{s}_{p,q}}+\|z^-_0\|_{\dot{F}^{s}_{p,q}}\right).
		\end{equation}
		By Lemma \ref{Bernstein}, Remark \ref{remark1} and the fact $F^{s-1}_{p,q}\hookrightarrow L^p$ when $s>1$, one gets 
		\begin{align*}
			\|\Delta_{n+1}z^+_0\|_{L^p}+\|\Delta_{n+1}z^-_0\|_{L^p}
			&\leq C2^{-(n+1)}\left( \|\nabla z^+_0\|_{L^p}+\|\nabla z^-_0\|_{L^p}\right)\\
			&\leq C2^{-(n+1)}\left( \|z^+_0\|_{F^{s}_{p,q}}+\|z^-_0\|_{F^{s}_{p,q}}\right).
		\end{align*}
		Thus, we have
		\begin{equation}\label{equation6}
			\|\Delta_{n+1}z^+_0\|_{F^{s-1}_{p,q}}+\|\Delta_{n+1}z^-_0\|_{F^{s-1}_{p,q}}\leq C2^{-n}\left( \|z^+_0\|_{F^{s}_{p,q}}+\|z^-_0\|_{F^{s}_{p,q}}\right).
		\end{equation}
		Taking the divergence on both sides of \eqref{IMHDCauchy}, one has
		\begin{equation}
			\delta\pi_1^{(n+1)}=\partial_j(-\Delta)^{-1}(\delta z_i^{-^{(n)}}\partial_iz_j^{+^{(n)}})+\partial_i(-\Delta)^{-1}(\partial_j z_i^{-^{(n)}}\delta z_j^{+^{(n+1)}}),
		\end{equation}
		\begin{equation}
			\delta\pi_2^{(n+1)}=\partial_j(-\Delta)^{-1}(\delta z_i^{+^{(n)}}\partial_iz_j^{-^{(n)}})+\partial_i(-\Delta)^{-1}(\partial_j z_i^{+^{(n)}}\delta z_j^{-^{(n+1)}}).
		\end{equation}
		Similar to the proof of \eqref{pi1}, we have
		\begin{align}
			&\mathrel{\phantom{=}}\|\nabla\delta\pi_1^{(n+1)}\|_{F^{s-1}_{p,q}}+\|\nabla\delta\pi_2^{(n+1)}\|_{F^{s-1}_{p,q}}\nonumber\\
			&\leq C\left(\|\delta z^{-^{(n)}}\|_{F^{s-1}_{p,q}}\| z^{+^{(n)}}\|_{F^{s}_{p,q}}+\| z^{-^{(n)}}\|_{F^{s}_{p,q}}\| \delta z^{+^{(n+1)}}\|_{F^{s-1}_{p,q}}\right)\nonumber\\
			&\mathrel{\phantom{=}}+ C\left(\|\delta z^{+^{(n)}}\|_{F^{s-1}_{p,q}}\| z^{-^{(n)}}\|_{F^{s}_{p,q}}+\| z^{+^{(n)}}\|_{F^{s}_{p,q}}\| \delta z^{-^{(n+1)}}\|_{F^{s-1}_{p,q}}\right)\label{equation7}.
		\end{align}
		Applying Proposition \ref{moser-type-estimate} and the embedding $F^{s-1}_{p,q}\hookrightarrow L^\infty$, one deduces
		\begin{align}
			&\mathrel{\phantom{=}}\|\delta z^{-^{(n)}}\cdot\nabla z^{+^{(n)}}\|_{F^{s-1}_{p,q}}+\|\delta z^{+^{(n)}}\cdot\nabla z^{-^{(n)}}\|_{F^{s-1}_{p,q}}\nonumber\\
			&\leq C\left(\|\delta z^{-^{(n)}}\|_{F^{s-1}_{p,q}}\| z^{+^{(n)}}\|_{F^{s}_{p,q}}+\| z^{-^{(n)}}\|_{F^{s}_{p,q}}\|\delta z^{+^{(n)}}\|_{F^{s-1}_{p,q}}\right).\label{equation8}
		\end{align}
		Hence, follows from \eqref{equation5}, \eqref{equation6}, \eqref{equation7} and \eqref{equation8} that
		\begin{align*}		
			&\mathrel{\phantom{=}}\|\delta z^{+^{(n+1)}}\|_{F^{s-1}_{p,q}}+\|\delta z^{-^{(n+1)}}\|_{F^{s-1}_{p,q}}\\
			&\leq C\Big(2^{-n}(\|z^{+}_0\|_{F^{s}_{p,q}}+\|z^{-}_0\|_{F^{s}_{p,q}})\\
			&\mathrel{\phantom{=}}+T\sup_{t\in[0,T]}\big(\|\delta z^{-^{(n)}}\|_{F^{s-1}_{p,q}}\| z^{+^{(n)}}\|_{F^{s}_{p,q}}+\| z^{-^{(n)}}\|_{F^{s}_{p,q}}\| \delta z^{+^{(n+1)}}\|_{F^{s-1}_{p,q}}\\
			&\mathrel{\phantom{=}}+ \|\delta z^{+^{(n)}}\|_{F^{s-1}_{p,q}}\| z^{-^{(n)}}\|_{F^{s}_{p,q}}+\| z^{+^{(n)}}\|_{F^{s}_{p,q}}\| \delta z^{-^{(n+1)}}\|_{F^{s-1}_{p,q}}\big)\Big),
		\end{align*}
		which together with \eqref{IMHDyoujie} implies
		\begin{align*}		
			&\mathrel{\phantom{=}}\|\delta z^{+^{(n+1)}}\|_{L^\infty_T(F^{s-1}_{p,q}))}+\|\delta z^{-^{(n+1)}}\|_{L^\infty_T(F^{s-1}_{p,q})}\\
			&\leq C(\|z^+_0\|_{F^s_{p,q}}+\|z^-_0\|_{F^s_{p,q}})\Big(2^{-n}+T\big(\|\delta z^{-^{(n)}}\|_{L^\infty_T(F^{s-1}_{p,q})}+\| \delta z^{+^{(n+1)}}\|_{L^\infty_T(F^{s-1}_{p,q})}\\
			&\quad+ \|\delta z^{+^{(n)}}\|_{L^\infty_T(F^{s-1}_{p,q})}+\| \delta z^{-^{(n+1)}}\|_{L^\infty_T(F^{s-1}_{p,q})}\big)\Big).
		\end{align*}
		Thus, if $C(\|z^+_0\|_{F^s_{p,q}}+\|z^-_0\|_{F^s_{p,q}})T\leq\frac{1}{8}$, then
		\begin{align}	
			&\quad\|\delta z^{+^{(n+1)}}\|_{L^\infty_T(F^{s-1}_{p,q})}+\|\delta z^{-^{(n+1)}}\|_{L^\infty_T(F^{s-1}_{p,q})}\nonumber\\
			&\leq C(\|z^+_0\|_{F^s_{p,q}}+\|z^-_0\|_{F^s_{p,q}})2^{-n}\nonumber\\
			&\quad+2C(\|z^+_0\|_{F^s_{p,q}}+\|z^-_0\|_{F^s_{p,q}})T\big(\|\delta z^{-^{(n)}}\|_{L^\infty_T(F^{s-1}_{p,q})}+ \|\delta z^{+^{(n)}}\|_{L^\infty_T(F^{s-1}_{p,q})}\big),
		\end{align}
		which yields that 
		\begin{equation}\label{equation9}		
			\|\delta z^{+^{(n+1)}}\|_{L^\infty_T(F^{s-1}_{p,q})}+\|\delta z^{-^{(n+1)}}\|_{L^\infty_T(F^{s-1}_{p,q})}\leq 2C(\|z^+_0\|_{F^s_{p,q}}+\|z^-_0\|_{F^s_{p,q}})2^{-n}.			
		\end{equation}
		Therefore, $(z^{+^{(n)}},z^{-^{(n)}})_{n\in\mathbb{N}}$ is a Cauchy sequence in $C([0,T];F^{s-1}_{p,q}\times F^{s-1}_{p,q})$, whence it converges to some limit function $(z^+,z^-)\in C([0,T];F^{s-1}_{p,q}\times F^{s-1}_{p,q})$. \\
		\textbf{Step 4. Conclusion.}
		
		Finally, we prove the limit $(z^+,z^-)\in E^{s}_{p,q}(T)\times E^{s}_{p,q}(T)$ for all $T$ satisfying
		\[ 0<T\leq\min\{T_0,\frac{1}{8C(\|z^+_0\|_{F^s_{p,q}}+\|z^-_0\|_{F^s_{p,q}})}\}, \]
		and that it satisfies the system \eqref{yongMHD}. Indeed, from Step 2, $(z^{+^{(n)}},z^{-^{(n)}})_{n\in\mathbb{N}}$ is uniformly bounded in $L^\infty(0,T;F^s_{p,q}\times F^{s}_{p,q})$. Then the Fatou property (Lemma \ref{lem:Triebel-Lizorkin-properties} $(iii)$) for Triebel-Lizorkin spaces guarantees that $(z^+,z^-)\in L^\infty(0,T;F^s_{p,q}\times F^{s}_{p,q})$. Moreover, $(z^+,z^-)$ satisfies
		\begin{equation}\label{jieyoujie}
			\|z^+\|_{L^\infty_T(F^s_{p,q})}+\|z^-\|_{L^\infty_T(F^s_{p,q})}\leq 2C(\|z^+_0\|_{F^s_{p,q}}+\|z^-_0\|_{F^s_{p,q}}).
		\end{equation}
		As shown in Step 3, $(z^{+^{(n)}},z^{-^{(n)}})_{n\in\mathbb{N}}$ converges to $(z^+,z^-)$ in $C([0,T];F^{s-1}_{p,q}\times F^{s-1}_{p,q})$, an interpolation argument ensures that convergence actually holds true in $C([0,T];F^{s'}_{p,q}\times F^{s'}_{p,q})$ for any $s'<s$. It is easy to pass to the limit in system \eqref{yongMHD} and to conclude that $(u,b)$ is indeed a solution to system \eqref{MHD} with the initial data $(u_0,b_0)\in F^s_{p,q}\times F^{s}_{p,q}$.
		
		On the other hand, since $(z^+,z^-)\in L^\infty(0,T;F^s_{p,q}\times F^{s}_{p,q})$, it follows from the system \eqref{yongMHD} that $\nabla\pi\in L^\infty(0,T;F^s_{p,q})$. In the case $q<\infty$, Theorem \ref{shuyunfangchengcunzaixing} enables us to conclude that $(z^+,z^-)\in C([0,T];F^s_{p,q}\times F^{s}_{p,q})$. Finally using the system \eqref{yongMHD} again, we see that $\partial_t z^+$ and $\partial_tz^-$ are both in $C([0,T];F^{s-1}_{p,q})$ if $q$ is finite, and in $L^\infty(0,T;F^{s-1}_{p,q})$ otherwise. Hence, the solution $(z^+,z^-)$ belongs to $E^s_{p,q}(T)\times E^s_{p,q}(T)$.
		
		\textbf{Uniqueness:} Let us consider that $(z^{+'},z^{-'})\in C([0,T];F^s_{p,q}\times F^{s}_{p,q})$ is another solution to the system \eqref{yongMHD} with the same initial data. Denote $\delta z^+=z^+-z^{+'}$ and $\delta z^-=z^--z^{-'}$, then we find 
		\begin{equation}
			\begin{cases*}
				\partial_t\delta z^++(z^-\cdot\nabla)\delta z^+=-(\delta z^-\cdot\nabla)z^+-\nabla(\pi-\pi'),\\
				\partial_t\delta z^-+(z^+\cdot\nabla)\delta z^-=-(\delta z^+\cdot\nabla)z^--\nabla(\pi-\pi'),\\
				\nabla\cdot\delta z^+=\nabla\cdot\delta z^-=0,\\
				\delta z^{+}(0)=\delta z^{-}(0)=0.
			\end{cases*}
		\end{equation}
		A similar argument to that used in the derivation of \eqref{equation9} yields
		\begin{equation}		
			\|\delta z^{+}\|_{E^{s-1}_{p,q}(T)}+\|\delta z^{-}\|_{E^{s-1}_{p,q}(T)}\leq C(\|\delta z^{+}(0)\|_{F^{s-1}_{p,q}}+\|\delta z^{-}(0)\|_{F^{s-1}_{p,q}})=0,
		\end{equation}
		which ensures the uniqueness of the solution to system \eqref{yongMHD}.
	
		\textbf{Continuity (continuous dependence of the solution map):} For any $z^+_0,z^-_0,\tilde{z}_0^{+'},\tilde{z}_0^{-'}\in D(R)\triangleq\{f\in F^s_{p,q}\colon \|f\|_{F^s_{p,q}}\leq R,\textup{div}f=0\}$, we denote the corresponding solutions $z^+=\mathfrak{S}_T(z^+_0),z^-=\mathfrak{S}_T(z^-_0),\tilde{z}^{+'}=\mathfrak{S}_T(\tilde{z}^{+'}_0),\tilde{z}^{-'}=\mathfrak{S}_T(\tilde{z}^{-'}_0)$, and set $w^+=z^+-\tilde{z}^{+'},w^-=z^--\tilde{z}^{-'}$, then $(w^+,w^-)$ solves
		\begin{equation}
			\begin{cases*}
				\partial_tw^++(z^-\cdot\nabla)w^+=-(w^-\cdot\nabla)z^+-\nabla(\pi-\tilde{\pi}'),\\
				\partial_tw^-+(z^+\cdot\nabla)w^-=-(w^+\cdot\nabla)z^--\nabla(\pi-\tilde{\pi}'),\\
				\nabla\cdot w^+=\nabla\cdot w^-=0,\\
				w^{+}(0)=z_0^+-\tilde{z}_0^{+'},\\
				w^{-}(0)=z_0^--\tilde{z}_0^{-'}.
			\end{cases*}
		\end{equation}
		By means of \eqref{feiqiciF-priori-estimates-divv=0} in Theorem \ref{xianyanguji} and the embedding $F^{s-1}_{p,q}\hookrightarrow L^\infty$ again, we find
		\begin{align}
			&\quad\|w^+\|_{F^{s-1}_{p,q}}+\|w^-\|_{F^{s-1}_{p,q}}\nonumber\\
			&\leq \|w^+_0\|_{F^{s-1}_{p,q}}+\|w^-_0\|_{F^{s-1}_{p,q}}+\int_{0}^t(\|w^-\cdot\nabla z^+\|_{F^{s-1}_{p,q}}+\|w^+\cdot\nabla z^-\|_{F^{s-1}_{p,q}})d\tau\nonumber\\
			&\quad+2\int_{0}^t\|\nabla(\pi-\tilde{\pi}')\|_{F^{s-1}_{p,q}}d\tau+\int_0^t\big(\| z^-\|_{F^{s}_{p,q}}\|w^+\|_{F^{s-1}_{p,q}}+\| z^+\|_{F^{s}_{p,q}}\|w^-\|_{F^{s-1}_{p,q}}\big)d\tau.\label{equation24}
		\end{align}
		Using Proposition \ref{moser-type-estimate} and the embedding $F^{s-1}_{p,q}\hookrightarrow L^\infty$ again, we obtain
		\begin{equation}
			\|w^-\cdot\nabla z^+\|_{F^{s-1}_{p,q}}+\|w^+\cdot\nabla z^-\|_{F^{s-1}_{p,q}}\leq C\big(\|w^-\|_{F^{s-1}_{p,q}}\|\nabla z^+\|_{F^s_{p,q}}+\|w^+\|_{F^{s-1}_{p,q}}\|\nabla z^-\|_{F^s_{p,q}}\big) .
		\end{equation}
		Notice that
		\begin{equation*}
			\nabla(\pi-\tilde{\pi}')=\nabla(-\Delta)^{-1}\text{div}(z^-\cdot\nabla w^++w^-\cdot\nabla z^+).
		\end{equation*}
		Similar to the proof of \eqref{pi1}, one gets 
		\begin{equation}\label{equation25}
			\|\nabla(\pi-\tilde{\pi}')\|_{F^{s-1}_{p,q}}\leq C\big(\|w^-\|_{F^{s-1}_{p,q}}\|z^+\|_{F^s_{p,q}}+\|w^+\|_{F^{s-1}_{p,q}}\|z^-\|_{F^s_{p,q}}\big).
		\end{equation}
		Summarizing the above estimates \eqref{equation24}-\eqref{equation25}, we deduce
		\begin{align*}
			&\quad\|w^+\|_{F^{s-1}_{p,q}}+\|w^-\|_{F^{s-1}_{p,q}}\\
			&\leq \|w^+_0\|_{F^{s-1}_{p,q}}+\|w^-_0\|_{F^{s-1}_{p,q}}+C\int_0^t\big(\| z^-\|_{F^{s}_{p,q}}+\| z^+\|_{F^{s}_{p,q}}\big)\big(\|w^+\|_{F^{s-1}_{p,q}}+\|w^-\|_{F^{s-1}_{p,q}}\big)d\tau.
		\end{align*}
		Applying Gronwall's inequality and \eqref{jieyoujie}, one obtains
		\begin{align}
			&\quad\|\mathfrak{S}_T(z^+_0)-\mathfrak{S}_T(\tilde{z}^{+'}_0)\|_{E^{s-1}_{T}}+\|\mathfrak{S}_T(z^-_0)-\mathfrak{S}_T(\tilde{z}^{-'}_0)\|_{E^{s-1}_{T}}\nonumber\\
			&\leq C(\|z_0^+-\tilde{z}_0^{+'}\|_{F^{s-1}_{p,q}}+\|z_0^--\tilde{z}_0^{-'}\|_{F^{s-1}_{p,q}})\label{equation10}.
		\end{align}
		Thus, \eqref{equation10} combined with an obvious interpolation ensures the continuity with respect to the initial data in $C([0,T];F^{s'}_{p,q}\times F^{s'}_{p,q})$ for any $s'<s$.
		
		In the case of $q<\infty$, from \eqref{S_jdingyi}, we denote the corresponding solution $z^{+^{N}}=\mathfrak{S}_T(S_{N+1}z^+_0)$, $z^{-^{N}}=\mathfrak{S}_T(S_{N+1}z^-_0)$ and set $ w^{+^N}=z^+-z^{+^{N}}$, $w^{-^N}=z^--z^{-^{N}}$. 
		%It is not hard to see that 
		Apparently, $(w^{+^N},w^{-^N})$ solves the following system:
		\begin{equation}
			\begin{cases*}
				\partial_tw^{+^N}+(z^-\cdot\nabla)w^{+^N}=-(w^{-^{N}}\cdot\nabla)z^+-\nabla(\pi-\pi^N),\\
				\partial_tw^{-^N}+(z^+\cdot\nabla)w^{-^N}=-(w^{+^N}\cdot\nabla)z^--\nabla(\pi-\pi^N),\\
				\nabla\cdot w^{+^N}=\nabla\cdot w^{-^N}=0,\\
				w^{+^N}(0)=z_0^+-S_{N+1}z_0^{+},\\
				w^{-^N}(0)=z_0^--S_{N+1}z_0^{-}.
			\end{cases*}
		\end{equation}
		In a manner analogous to the proof of \eqref{equation10}, we have
		\begin{align}
			&\quad\|\mathfrak{S}_T(z^+_0)-\mathfrak{S}_T(S_{N+1}z_0^{+})\|_{E^{s}_{T}}+\|\mathfrak{S}_T(z_0^{-})-\mathfrak{S}_T(S_{N+1}z_0^{-})\|_{E^{s}_{T}}\nonumber\\
			&\leq C(\|z_0^+-S_{N+1}z_0^{+}\|_{F^{s}_{p,q}}+\|z_0^--S_{N+1}z_0^{-}\|_{F^{s}_{p,q}}).\label{equation11}
		\end{align}
		
		Now, we show the continuity of the
		solution map in $C([0,T];F^s_{p,q}\times F^s_{p,q})$ as $1\leq p,q<\infty$. Let $\tilde{z}^+,\tilde{z}^-\in D(R)$. By virtue of \eqref{equation10} \eqref{equation11} and  Lemma \ref{lem:Triebel-Lizorkin-properties} $(iv)$, we infer
		\begin{align*}
			&\quad\|\mathfrak{S}_T(z^+_0)-\mathfrak{S}_T(\tilde{z}^{+}_0)\|_{L^\infty_T(F^{s}_{p,q})}+\|\mathfrak{S}_T(z^-_0)-\mathfrak{S}_T(\tilde{z}^{-}_0)\|_{L^\infty_T(F^{s}_{p,q})}\\
			&\leq \|\mathfrak{S}_T(z^+_0)-\mathfrak{S}_T(S_{N+1}z^{+}_0)\|_{L^\infty_T(F^{s}_{p,q})}+\|\mathfrak{S}_T(z^-_0)-\mathfrak{S}_T(S_{N+1}z^{-}_0)\|_{L^\infty_T(F^{s}_{p,q})}\\
			&\quad+\|\mathfrak{S}_T(\tilde{z}^{+}_0)-\mathfrak{S}_T(S_{N+1}\tilde{z}^{+}_0)\|_{L^\infty_T(F^{s}_{p,q})}+\|\mathfrak{S}_T(\tilde{z}^{-}_0)-\mathfrak{S}_T(S_{N+1}\tilde{z}^{-}_0)\|_{L^\infty_T(F^{s}_{p,q})}\\
			&\quad+\|\mathfrak{S}_T(S_{N+1}z^{+}_0)-\mathfrak{S}_T(S_{N+1}\tilde{z}^{+}_0)\|_{L^\infty_T(F^{s}_{p,q})}+\|\mathfrak{S}_T(S_{N+1}z^{-}_0)-\mathfrak{S}_T(S_{N+1}\tilde{z}^{-}_0)\|_{L^\infty_T(F^{s}_{p,q})}\\
			&\leq C\big( \|z^+_0-S_{N+1}z^{+}_0\|_{F^{s}_{p,q}}+\|\tilde{z}^{+}_0-S_{N+1}\tilde{z}^{+}_0\|_{F^{s}_{p,q}}+\|z^-_0-S_{N+1}z^{-}_0\|_{F^{s}_{p,q}}+\|\tilde{z}^{-}_0-S_{N+1}\tilde{z}^{-}_0\|_{F^{s}_{p,q}}\big)\\
			&\quad+\|\mathfrak{S}_T(S_{N+1}z^{+}_0)-\mathfrak{S}_T(S_{N+1}\tilde{z}^{+}_0)\|^{\frac{1}{2}}_{L^\infty_T(F^{s-1}_{p,q})}\|\mathfrak{S}_T(S_{N+1}z^{+}_0)-\mathfrak{S}_T(S_{N+1}\tilde{z}^{+}_0)\|^{\frac{1}{2}}_{L^\infty_T(F^{s+1}_{p,q})}\\
			&\quad+\|\mathfrak{S}_T(S_{N+1}z^{-}_0)-\mathfrak{S}_T(S_{N+1}\tilde{z}^{-}_0)\|^{\frac{1}{2}}_{L^\infty_T(F^{s-1}_{p,q})}\|\mathfrak{S}_T(S_{N+1}z^{-}_0)-\mathfrak{S}_T(S_{N+1}\tilde{z}^{-}_0)\|^{\frac{1}{2}}_{L^\infty_T(F^{s+1}_{p,q})}\\
			&\leq C\big( \|z^+_0-S_{N+1}z^{+}_0\|_{F^{s}_{p,q}}+\|\tilde{z}^{+}_0-S_{N+1}\tilde{z}^{+}_0\|_{F^{s}_{p,q}}+\|z^-_0-S_{N+1}z^{-}_0\|_{F^{s}_{p,q}}+\|\tilde{z}^{-}_0-S_{N+1}\tilde{z}^{-}_0\|_{F^{s}_{p,q}}\big)\\
			&\quad+C2^{\frac{N}{2}}R^{\frac{1}{2}}(\|z_0^+-\tilde{z}_0^{+}\|_{F^{s-1}_{p,q}}^{\frac{1}{2}}+\|z_0^--\tilde{z}_0^{-}\|_{F^{s-1}_{p,q}}^{\frac{1}{2}}),
		\end{align*}
		where we used Lemma 
		\ref{lem:Snguji} in the last inequality. Since $1\leq p,q<\infty$, then for any $\varepsilon>0$, one can select $N$ to be sufficiently large, such that
		\begin{equation*}
			C\big( \|z^+_0-S_{N+1}z^{+}_0\|_{F^{s}_{p,q}}+\|\tilde{z}^{+}_0-S_{N+1}\tilde{z}^{+}_0\|_{F^{s}_{p,q}}+\|z^-_0-S_{N+1}z^{-}_0\|_{F^{s}_{p,q}}+\|\tilde{z}^{-}_0-S_{N+1}\tilde{z}^{-}_0\|_{F^{s}_{p,q}}\big)\leq \frac{\varepsilon}{2}.
		\end{equation*}
		Then we choose $\sigma$ small enough such that $\|z_0^+-\tilde{z}_0^{+}\|_{F^{s}_{p,q}},\|z_0^--\tilde{z}_0^{-}\|_{F^{s}_{p,q}}<\sigma$ and $C2^{\frac{N}{2}}R^{\frac{1}{2}}\sigma^{\frac{1}{2}}< \frac{\varepsilon}{4}$. Hence, we have
		\begin{equation*}
			\|\mathfrak{S}_T(z^+_0)-\mathfrak{S}_T(\tilde{z}^{+}_0)\|_{L^\infty_T(F^{s}_{p,q})}+\|\mathfrak{S}_T(z^-_0)-\mathfrak{S}_T(\tilde{z}^{-}_0)\|_{L^\infty_T(F^{s}_{p,q})}\leq \varepsilon.
		\end{equation*}
		 This yields the continuous dependence of the solution map. Therefore, we complete the proof of Theorem \ref{MHD-Local-well-posedness}.	 
	\end{proof}

	\begin{proof}[Proof of Theorem \ref{Blow-up Criterion}]
		Applying \eqref{feiqiciF-priori-estimates} of Theorem \ref{xianyanguji} to the system \eqref{yongMHD}, one gets
		\begin{align}
			\|z^+\|_{F^s_{p,q}}+\|z^-\|_{F^s_{p,q}}\leq&\|z_0^+\|_{F^s_{p,q}}+\|z_0^-\|_{F^s_{p,q}}+2\int_{0}^t\|\nabla\pi\|_{F^s_{p,q}}d\tau\nonumber\\
			&+C\int^t_0(\|z^+\|_{F^s_{p,q}}+\|z^-\|_{F^s_{p,q}})(\|\nabla z^+\|_{L^\infty}+\|\nabla z^-\|_{L^\infty})d\tau.\label{equation15}
		\end{align}
		Owing to the boundedness of operator $\partial_j\partial_k(-\Delta)^{-1}$ in $\dot{F}^s_{p,q}$ (see \cite{MR1009119}), Proposition \ref{moser-type-estimate} and Lemma \ref{tiduheqiciF}, one can show that if $(p,q)\in[1,\infty)\times[1,\infty]$ or $p=q=\infty$,
		\begin{align}
			\|\nabla \pi\|_{\dot{F}^{s}_{p,q}}&\leq C\|\nabla\nabla(\Delta)^{-1}\text{div}(z^{-}\cdot\nabla)z^+\|_{\dot{F}^{s-1}_{p,q}}\nonumber\\
			&\leq C\big(\|\nabla z^-\|_{L^\infty}\| z^+\|_{\dot{F}^{s}_{p,q}}+\|z^-\|_{\dot{F}^{s}_{p,q}}\|\nabla z^+\|_{L^\infty}\big). \label{equation16}
		\end{align}
		By Lemma \ref{p=1Reszi}, one infers when $1\leq p\leq\infty$,
		\begin{align}
			2\|S_0\nabla\pi\|_{L^p}&=\|S_0\nabla(\Delta)^{-1}\text{div}(z^{-}\cdot\nabla)z^+\|_{L^p}+\|S_0\nabla(\Delta)^{-1}\text{div}(z^{+}\cdot\nabla)z^-\|_{L^p}\nonumber\\
			&\leq C\big(\|z^-\|_{L^p}\|z^+\|_{L^\infty}+\|z^+\|_{L^p}\|z^-\|_{L^\infty}\big).\label{equation17}
		\end{align}
		Combining \eqref{equation15}, \eqref{equation16}, and \eqref{equation17}, we obtain
		\begin{align}
			\|z^+\|_{F^s_{p,q}}+\|z^-\|_{F^s_{p,q}}\leq&\|z_0^+\|_{F^s_{p,q}}+\|z_0^-\|_{F^s_{p,q}}+C\int^t_0(\|z^+\|_{F^s_{p,q}}+\|z^-\|_{F^s_{p,q}})\nonumber\\
			&\times(\| z^+\|_{L^\infty}+\|z^-\|_{L^\infty}+\|\nabla z^+\|_{L^\infty}+\|\nabla z^-\|_{L^\infty})d\tau.\label{equation18}
		\end{align}
		Then Gronwall's inequality yields the desired blow-up criterion \eqref{blow-up1}.
		
		On the other hand, we recall the fact that the
		elliptic system, $\text{div }v=0$ and $\nabla\times v=w$ implies 
		\[ \nabla v=\mathcal{P}(w)+\mathcal{M}w, \]
		where $\mathcal{P}$ is a singular integral operator homogeneous of degree $-n$, and $\mathcal{M}$ is a constant matrix. By the boundedness of singular integral operator From $\dot{F}^0_{\infty,\infty}$ into itself \cite{MR1009119}, one gets 
		\begin{equation}
			\|\nabla z^-\|_{\dot{F}^0_{\infty,\infty}}\leq C\|\nabla\times z^-\|_{\dot{F}^0_{\infty,\infty}},\quad\|\nabla z^+\|_{\dot{F}^0_{\infty,\infty}}\leq C\|\nabla\times z^+\|_{\dot{F}^0_{\infty,\infty}}.
		\end{equation}
		When $s>1+\frac{d}{p}$, by means of Lemma \ref{lem:logarithmic-Triebel-Lizorkin-space-inequality}, we have
		\begin{align}
			\|\nabla z^+\|_{L^\infty}&\leq C\big(1+\|\nabla z^+\|_{\dot{F}^0_{\infty,\infty}}(\ln^+\|\nabla z^+\|_{F^{s-1}_{p,q}}+1)\big)\nonumber\\
			&\leq  C\big(1+\|\nabla \times z^+\|_{\dot{F}^0_{\infty,\infty}}(\ln^+\| z^+\|_{F^{s}_{p,q}}+1)\big)\label{equation22},
		\end{align}
		and
		\begin{align}
			\|\nabla z^-\|_{L^\infty}&\leq C\big(1+\|\nabla z^-\|_{\dot{F}^0_{\infty,\infty}}(\ln^+\|\nabla z^-\|_{F^{s-1}_{p,q}}+1)\big)\nonumber\\
			&\leq  C\big(1+\|\nabla \times z^-\|_{\dot{F}^0_{\infty,\infty}}(\ln^+\| z^-\|_{F^{s}_{p,q}}+1)\big)\label{equation23}.
		\end{align}
		Substituting the two above estimates \eqref{equation22} and \eqref{equation23}  into \eqref{equation18}, we obtain 
		\begin{align}
			&\quad\|z^+\|_{F^s_{p,q}}+\|z^-\|_{F^s_{p,q}}\nonumber\\
			&\leq\|z_0^+\|_{F^s_{p,q}}+\|z_0^-\|_{F^s_{p,q}}+C\int^t_0(\|z^+\|_{F^s_{p,q}}+\|z^-\|_{F^s_{p,q}})\nonumber\\
			&\quad\times\big((1+\| z^+\|_{L^\infty}+\|z^-\|_{L^\infty}+\|\nabla \times z^+\|_{\dot{F}^0_{\infty,\infty}}+\|\nabla\times z^-\|_{\dot{F}^0_{\infty,\infty}})\nonumber\\
			&\quad\times(1+\ln^+\|z^+\|_{F^s_{p,q}}+\ln^+\|z^-\|_{F^s_{p,q}})\big)d\tau,
		\end{align}
		which along with the Osgood inequality in \cite{MR2768550} implies
		\begin{align*}
			&\|z^+\|_{F^s_{p,q}}+\|z^-\|_{F^s_{p,q}}\leq\big( \|z^+_0\|_{F^s_{p,q}}+\|z^-_0\|_{F^s_{p,q}}\big)\times\\
			&\textup{exp}\bigg[C\textup{exp}\big[C\int_{0}^{t}(1+\| z^+\|_{L^\infty}+\|z^-\|_{L^\infty}+\|(\nabla\times z^+)(t)\|_{\dot{F}^0_{\infty,\infty}}+\|(\nabla\times z^-)(t)\|_{\dot{F}^0_{\infty,\infty}})d\tau\big]\bigg].
		\end{align*}
		This leads to the blow-up criterion \eqref{blow-up2}.
		
		In the case of $(p,q)\in(1,\infty)\times[1,\infty]$ and $s>1+\frac{d}{p}$, motivated by \cite{MR2592288}, by the $L^p$-boundedness of Riesz transform and $\pi=(-\Delta)^{-1}\partial_i\partial_j(z^-_iz^+_j)$, one can see
		\begin{equation}\label{equation20}
			2\|\nabla \pi\|_{L^p}\leq C\big(\|\nabla z^+\|_{L^\infty}\|z^-\|_{L^p}+\|\nabla z^-\|_{L^\infty}\|z^+\|_{L^p}\big),\quad 1<p<\infty.
		\end{equation}
		A combination of \eqref{equation15}, \eqref{equation16} and \eqref{equation20} yields
		\begin{align}
			&\quad\|z^+\|_{F^s_{p,q}}+\|z^-\|_{F^s_{p,q}}\nonumber\\
			&\leq\|z_0^+\|_{F^s_{p,q}}+\|z_0^-\|_{F^s_{p,q}}+C\int^t_0(\|z^+\|_{F^s_{p,q}}+\|z^-\|_{F^s_{p,q}})(\|\nabla z^+\|_{L^\infty}+\|\nabla z^-\|_{L^\infty})d\tau.\label{equation21}
		\end{align}
		Then, substituting \eqref{equation22} and \eqref{equation23} into \eqref{equation21}, and applying the Osgood inequality \cite{MR2768550} again, one obtains
		\begin{equation*}
			\|z^+\|_{F^s_{p,q}}+\|z^-\|_{F^s_{p,q}}\leq\big( \|z^+_0\|_{F^s_{p,q}}+\|z^-_0\|_{F^s_{p,q}}\big)e^{Ce^{C\int_{0}^{t}(1+\|(\nabla\times z^+)(t)\|_{\dot{F}^0_{\infty,\infty}}+\|(\nabla\times z^-)(t)\|_{\dot{F}^0_{\infty,\infty}})d\tau}},
		\end{equation*}
		which implies the blow-up criterion \eqref{blow-up3}. Therefore, we complete the proof of Theorem \ref{Blow-up Criterion}.
	\end{proof}

	\noindent{\bf Acknowledgments.} This work was partially supported by the National Natural Science Foundation of China under grant 11971188.

	\bibliographystyle{abbrv}
	\bibliography{ref}

@article {MR4240785,
	AUTHOR = {Guo, Zihua and Li, Kuijie},
	TITLE = {Remarks on the well-posedness of the {E}uler equations in the
	{T}riebel-{L}izorkin spaces},
	JOURNAL = {J. Fourier Anal. Appl.},
	FJOURNAL = {The Journal of Fourier Analysis and Applications},
	VOLUME = {27},
	YEAR = {2021},
	NUMBER = {2},
	PAGES = {Paper No. 29, 24},
	ISSN = {1069-5869,1531-5851},
	MRCLASS = {35Q31 (42B37 46E30 76B03)},
	MRNUMBER = {4240785},
}

@article{c6f44d5b-9c43-3ae5-9150-1b0704d03ff8,
	ISSN = {00029327, 10806377},
	author = {C. Fefferman and E. M. Stein},
	journal = {American Journal of Mathematics},
	number = {1},
	pages = {107--115},
	publisher = {The Johns Hopkins University Press},
	title = {Some Maximal Inequalities},
	urldate = {2025-02-25},
	volume = {93},
	year = {1971}
}

@book{Stein+1971,
	title = {Singular Integrals and Differentiability Properties of Functions},
	author = {Elias M. Stein},
	publisher = {Princeton University Press},
	address = {Princeton},
	isbn = {9781400883882},
	year = {1971},
	lastchecked = {2025-03-01}
}

@article {MR1880646,
	AUTHOR = {Chae, Dongho},
	TITLE = {On the well-posedness of the {E}uler equations in the
	{T}riebel-{L}izorkin spaces},
	JOURNAL = {Comm. Pure Appl. Math.},
	FJOURNAL = {Communications on Pure and Applied Mathematics},
	VOLUME = {55},
	YEAR = {2002},
	NUMBER = {5},
	PAGES = {654--678},
	ISSN = {0010-3640,1097-0312},
	MRCLASS = {35Q35 (35B30 76B03)},
	MRNUMBER = {1880646},
	MRREVIEWER = {Rodolfo\ Salvi},
}

@book {MR781540,
	AUTHOR = {Triebel, Hans},
	TITLE = {Theory of function spaces},
	SERIES = {Monographs in Mathematics},
	VOLUME = {78},
	PUBLISHER = {Birkh\"auser Verlag, Basel},
	YEAR = {1983},
	PAGES = {284},
	ISBN = {3-7643-1381-1},
	MRCLASS = {46Exx},
	MRNUMBER = {781540},
}

@article {MR2592288,
	AUTHOR = {Chen, Qionglei and Miao, Changxing and Zhang, Zhifei},
	TITLE = {On the well-posedness of the ideal {MHD} equations in the
	{T}riebel-{L}izorkin spaces},
	JOURNAL = {Arch. Ration. Mech. Anal.},
	FJOURNAL = {Archive for Rational Mechanics and Analysis},
	VOLUME = {195},
	YEAR = {2010},
	NUMBER = {2},
	PAGES = {561--578},
	ISSN = {0003-9527,1432-0673},
	MRCLASS = {35Q35 (35B60 76W05)},
	MRNUMBER = {2592288},
	MRREVIEWER = {Paolo\ Secchi},
}

@book {MR1232192,
	AUTHOR = {Stein, Elias M.},
	TITLE = {Harmonic analysis: real-variable methods, orthogonality, and
	oscillatory integrals},
	SERIES = {Princeton Mathematical Series},
	VOLUME = {43},
	PUBLISHER = {Princeton University Press, Princeton, NJ},
	YEAR = {1993},
	PAGES = {xiv+695},
	ISBN = {0-691-03216-5},
	MRCLASS = {42-02 (35Sxx 43-02 47G30)},
	MRNUMBER = {1232192},
	MRREVIEWER = {Michael\ Cowling},
}

@article {MR1009119,
	AUTHOR = {Frazier, M. and Torres, R. and Weiss, G.},
	TITLE = {The boundedness of {C}alder\'on-{Z}ygmund operators on the
	spaces {$\dot F^{\alpha,q}_p$}},
	JOURNAL = {Rev. Mat. Iberoamericana},
	FJOURNAL = {Revista Matem\'atica Iberoamericana},
	VOLUME = {4},
	YEAR = {1988},
	NUMBER = {1},
	PAGES = {41--72},
	ISSN = {0213-2230},
	MRCLASS = {42B20 (46E35 47G05)},
	MRNUMBER = {1009119},
	MRREVIEWER = {Cora\ Sadosky},
}

@book {MR1163193,
	AUTHOR = {Triebel, Hans},
	TITLE = {Theory of function spaces. {II}},
	SERIES = {Monographs in Mathematics},
	VOLUME = {84},
	PUBLISHER = {Birkh\"auser Verlag, Basel},
	YEAR = {1992},
	PAGES = {viii+370},
	ISBN = {3-7643-2639-5},
	MRCLASS = {46Exx (46-02)},
	MRNUMBER = {1163193},
	MRREVIEWER = {P.\ Szeptycki},
}

@book {MR1410258,
	AUTHOR = {Edmunds, D. E. and Triebel, H.},
	TITLE = {Function spaces, entropy numbers, differential operators},
	SERIES = {Cambridge Tracts in Mathematics},
	VOLUME = {120},
	PUBLISHER = {Cambridge University Press, Cambridge},
	YEAR = {1996},
	PAGES = {xii+252},
	ISBN = {0-521-56036-5},
	MRCLASS = {46E35 (35J70 35P15 35P20 47B06 47G30)},
	MRNUMBER = {1410258},
	MRREVIEWER = {Georgi\ E.\ Karadzhov},
}

@article {MR631751,
	AUTHOR = {Bony, Jean-Michel},
	TITLE = {Calcul symbolique et propagation des singularit\'es pour les
	\'equations aux d\'eriv\'ees partielles non lin\'eaires},
	JOURNAL = {Ann. Sci. \'Ecole Norm. Sup. (4)},
	FJOURNAL = {Annales Scientifiques de l'\'Ecole Normale Sup\'erieure.
	Quatri\`eme S\'erie},
	VOLUME = {14},
	YEAR = {1981},
	NUMBER = {2},
	PAGES = {209--246},
	ISSN = {0012-9593},
	MRCLASS = {35S99 (35G20)},
	MRNUMBER = {631751},
	MRREVIEWER = {J. Lacroix},
}

@book {MR2768550,
	AUTHOR = {Bahouri, Hajer and Chemin, Jean-Yves and Danchin, Rapha\"el},
	TITLE = {Fourier analysis and nonlinear partial differential equations},
	SERIES = {Grundlehren der mathematischen Wissenschaften [Fundamental
	Principles of Mathematical Sciences]},
	VOLUME = {343},
	PUBLISHER = {Springer, Heidelberg},
	YEAR = {2011},
	PAGES = {xvi+523},
	ISBN = {978-3-642-16829-1},
	MRCLASS = {35-02 (35L72 35Q30 42-02 42B37 76B03 76D03 76N10)},
	MRNUMBER = {2768550},
	MRREVIEWER = {Peter\ R.\ Massopust},
}

@book {MR1419319,
	AUTHOR = {Runst, Thomas and Sickel, Winfried},
	TITLE = {Sobolev spaces of fractional order, {N}emytskij operators, and
	nonlinear partial differential equations},
	SERIES = {De Gruyter Series in Nonlinear Analysis and Applications},
	VOLUME = {3},
	PUBLISHER = {Walter de Gruyter \& Co., Berlin},
	YEAR = {1996},
	PAGES = {x+547},
	ISBN = {3-11-015113-8},
	MRCLASS = {47H30 (35J65 46E35 46F10 47N20)},
	MRNUMBER = {1419319},
	MRREVIEWER = {P.\ Szeptycki},
	DOI = {10.1515/9783110812411},
	URL = {https://doi.org/10.1515/9783110812411},
}

@article {MR1022305,
	AUTHOR = {DiPerna, R. J. and Lions, P.-L.},
	TITLE = {Ordinary differential equations, transport theory and
	{S}obolev spaces},
	JOURNAL = {Invent. Math.},
	FJOURNAL = {Inventiones Mathematicae},
	VOLUME = {98},
	YEAR = {1989},
	NUMBER = {3},
	PAGES = {511--547},
	ISSN = {0020-9910,1432-1297},
	MRCLASS = {34A10 (34D20 35Q20 58D25 82A70)},
	MRNUMBER = {1022305},
	MRREVIEWER = {B.\ G.\ Pachpatte},
	DOI = {10.1007/BF01393835},
	URL = {https://doi.org/10.1007/BF01393835},
}

@article {MR2096794,
	AUTHOR = {Ambrosio, Luigi},
	TITLE = {Transport equation and {C}auchy problem for {$BV$} vector
	fields},
	JOURNAL = {Invent. Math.},
	FJOURNAL = {Inventiones Mathematicae},
	VOLUME = {158},
	YEAR = {2004},
	NUMBER = {2},
	PAGES = {227--260},
	ISSN = {0020-9910,1432-1297},
	MRCLASS = {35K15 (34A12 35L65)},
	MRNUMBER = {2096794},
	MRREVIEWER = {J.\ W.\ Jerome},
	DOI = {10.1007/s00222-004-0367-2},
	URL = {https://doi.org/10.1007/s00222-004-0367-2},
}

@article {MR2020259,
	AUTHOR = {Chae, Dongho},
	TITLE = {On the {E}uler equations in the critical {T}riebel-{L}izorkin
	spaces},
	JOURNAL = {Arch. Ration. Mech. Anal.},
	FJOURNAL = {Archive for Rational Mechanics and Analysis},
	VOLUME = {170},
	YEAR = {2003},
	NUMBER = {3},
	PAGES = {185--210},
	ISSN = {0003-9527,1432-0673},
	MRCLASS = {35Q35 (76B03)},
	MRNUMBER = {2020259},
	MRREVIEWER = {Marco\ Cannone},
	DOI = {10.1007/s00205-003-0271-8},
	URL = {https://doi.org/10.1007/s00205-003-0271-8},
}

@book {MR748308,
	AUTHOR = {Majda, A.},
	TITLE = {Compressible fluid flow and systems of conservation laws in
	several space variables},
	SERIES = {Applied Mathematical Sciences},
	VOLUME = {53},
	PUBLISHER = {Springer-Verlag, New York},
	YEAR = {1984},
	PAGES = {viii+159},
	ISBN = {0-387-96037-6},
	MRCLASS = {35L65 (76L05 76N10)},
	MRNUMBER = {748308},
	MRREVIEWER = {Joel\ Smoller},
	DOI = {10.1007/978-1-4612-1116-7},
	URL = {https://doi.org/10.1007/978-1-4612-1116-7},
}

@article {MR763762,
	AUTHOR = {Beale, J. T. and Kato, T. and Majda, A.},
	TITLE = {Remarks on the breakdown of smooth solutions for the {$3$}-{D}
	{E}uler equations},
	JOURNAL = {Comm. Math. Phys.},
	FJOURNAL = {Communications in Mathematical Physics},
	VOLUME = {94},
	YEAR = {1984},
	NUMBER = {1},
	PAGES = {61--66},
	ISSN = {0010-3616,1432-0916},
	MRCLASS = {35Q10 (76F99)},
	MRNUMBER = {763762},
	URL = {http://projecteuclid.org/euclid.cmp/1103941230},
}

@article {MR2299432,
	AUTHOR = {Cannone, Marco and Chen, Qionglei and Miao, Changxing},
	TITLE = {A losing estimate for the ideal {MHD} equations with
	application to blow-up criterion},
	JOURNAL = {SIAM J. Math. Anal.},
	FJOURNAL = {SIAM Journal on Mathematical Analysis},
	VOLUME = {38},
	YEAR = {2007},
	NUMBER = {6},
	PAGES = {1847--1859},
	ISSN = {0036-1410,1095-7154},
	MRCLASS = {35Q35 (35B40 35L60 76D03 76W05)},
	MRNUMBER = {2299432},
	MRREVIEWER = {Lorenzo\ Brandolese},
	DOI = {10.1137/060652002},
	URL = {https://doi.org/10.1137/060652002},
}

@article {MR1462753,
	AUTHOR = {Caflisch, Russel E. and Klapper, Isaac and Steele, Gregory},
	TITLE = {Remarks on singularities, dimension and energy dissipation for
	ideal hydrodynamics and {MHD}},
	JOURNAL = {Comm. Math. Phys.},
	FJOURNAL = {Communications in Mathematical Physics},
	VOLUME = {184},
	YEAR = {1997},
	NUMBER = {2},
	PAGES = {443--455},
	ISSN = {0010-3616,1432-0916},
	MRCLASS = {35Q30 (35D10 76W05)},
	MRNUMBER = {1462753},
	MRREVIEWER = {R.\ M.\ Gundersen},
	DOI = {10.1007/s002200050067},
	URL = {https://doi.org/10.1007/s002200050067},
}

@article {MR2173645,
	AUTHOR = {Zhang, Zhifei and Liu, Xiaofeng},
	TITLE = {On the blow-up criterion of smooth solutions to the 3{D} ideal
	{MHD} equations},
	JOURNAL = {Acta Math. Appl. Sin. Engl. Ser.},
	FJOURNAL = {Acta Mathematicae Applicatae Sinica. English Series},
	VOLUME = {20},
	YEAR = {2004},
	NUMBER = {4},
	PAGES = {695--700},
	ISSN = {0168-9673,1618-3932},
	MRCLASS = {35Q35 (35B40 76W05)},
	MRNUMBER = {2173645},
	DOI = {10.1007/s10255-004-0207-6},
	URL = {https://doi.org/10.1007/s10255-004-0207-6},
}

@article {MR2231013,
	AUTHOR = {Danchin, Rapha\"el},
	TITLE = {Estimates in {B}esov spaces for transport and
	transport-diffusion equations with almost {L}ipschitz
	coefficients},
	JOURNAL = {Rev. Mat. Iberoamericana},
	FJOURNAL = {Revista Matem\'atica Iberoamericana},
	VOLUME = {21},
	YEAR = {2005},
	NUMBER = {3},
	PAGES = {863--888},
	ISSN = {0213-2230},
	MRCLASS = {35K57 (35B45 35Q35)},
	MRNUMBER = {2231013},
	MRREVIEWER = {Guillaume\ Bal},
	DOI = {10.4171/RMI/438},
	URL = {https://doi.org/10.4171/RMI/438},
}

@article {MR752597,
	AUTHOR = {Alekseev, G. V.},
	TITLE = {Solvability of a homogeneous initial-boundary value problem
	for equations of magnetohydrodynamics of an ideal fluid},
	JOURNAL = {Dinamika Sploshn. Sredy},
	FJOURNAL = {Institut Gidrodinamiki Sibirskogo Otdeleniya Akademii Nauk
	SSSR. Dinamika Sploshno\u i\ Sredy},
	NUMBER = {57},
	YEAR = {1982},
	PAGES = {3--20},
	ISSN = {0420-0497},
	MRCLASS = {76W05 (35Q10)},
	MRNUMBER = {752597},
	MRREVIEWER = {Jaroslav\ Bart\'ak},
}

@article {MR2275873,
	AUTHOR = {Miao, Changxing and Yuan, Baoquan},
	TITLE = {Well-posedness of the ideal {MHD} system in critical {B}esov
	spaces},
	JOURNAL = {Methods Appl. Anal.},
	FJOURNAL = {Methods and Applications of Analysis},
	VOLUME = {13},
	YEAR = {2006},
	NUMBER = {1},
	PAGES = {89--106},
	ISSN = {1073-2772,1945-0001},
	MRCLASS = {35Q35 (35B30 76B03 76W05)},
	MRNUMBER = {2275873},
	MRREVIEWER = {Fabio\ V.\ Silva},
	DOI = {10.4310/MAA.2006.v13.n1.a5},
	URL = {https://doi.org/10.4310/MAA.2006.v13.n1.a5},
}

@article {MR1664597,
	AUTHOR = {Vishik, Misha},
	TITLE = {Hydrodynamics in {B}esov spaces},
	JOURNAL = {Arch. Ration. Mech. Anal.},
	FJOURNAL = {Archive for Rational Mechanics and Analysis},
	VOLUME = {145},
	YEAR = {1998},
	NUMBER = {3},
	PAGES = {197--214},
	ISSN = {0003-9527,1432-0673},
	MRCLASS = {35Q30 (46N20 76B03)},
	MRNUMBER = {1664597},
	MRREVIEWER = {Jos\'e\ Luiz\ Boldrini},
	DOI = {10.1007/s002050050128},
	URL = {https://doi.org/10.1007/s002050050128},
}

@book {MR1688875,
	AUTHOR = {Chemin, Jean-Yves},
	TITLE = {Perfect incompressible fluids},
	SERIES = {Oxford Lecture Series in Mathematics and its Applications},
	VOLUME = {14},
	NOTE = {Translated from the 1995 French original by Isabelle Gallagher
	and Dragos Iftimie},
	PUBLISHER = {The Clarendon Press, Oxford University Press, New York},
	YEAR = {1998},
	PAGES = {x+187},
	ISBN = {0-19-850397-0},
	MRCLASS = {76B47 (35-02 35Q35 76-02)},
	MRNUMBER = {1688875},
}

@article {MR951744,
	AUTHOR = {Kato, Tosio and Ponce, Gustavo},
	TITLE = {Commutator estimates and the {E}uler and {N}avier-{S}tokes
	equations},
	JOURNAL = {Comm. Pure Appl. Math.},
	FJOURNAL = {Communications on Pure and Applied Mathematics},
	VOLUME = {41},
	YEAR = {1988},
	NUMBER = {7},
	PAGES = {891--907},
	ISSN = {0010-3640,1097-0312},
	MRCLASS = {35Q10 (47F05 76D05)},
	MRNUMBER = {951744},
	MRREVIEWER = {Josef\ Bemelmans},
	DOI = {10.1002/cpa.3160410704},
	URL = {https://doi.org/10.1002/cpa.3160410704},
}

@article {MR1794270,
	AUTHOR = {Kozono, Hideo and Taniuchi, Yasushi},
	TITLE = {Limiting case of the {S}obolev inequality in {BMO}, with
	application to the {E}uler equations},
	JOURNAL = {Comm. Math. Phys.},
	FJOURNAL = {Communications in Mathematical Physics},
	VOLUME = {214},
	YEAR = {2000},
	NUMBER = {1},
	PAGES = {191--200},
	ISSN = {0010-3616,1432-0916},
	MRCLASS = {46E35 (35Q35 42B20 46N20 76B03)},
	MRNUMBER = {1794270},
	DOI = {10.1007/s002200000267},
	URL = {https://doi.org/10.1007/s002200000267},
}

@article {MR1980623,
	AUTHOR = {Kozono, Hideo and Ogawa, Takayoshi and Taniuchi, Yasushi},
	TITLE = {The critical {S}obolev inequalities in {B}esov spaces and
	regularity criterion to some semi-linear evolution equations},
	JOURNAL = {Math. Z.},
	FJOURNAL = {Mathematische Zeitschrift},
	VOLUME = {242},
	YEAR = {2002},
	NUMBER = {2},
	PAGES = {251--278},
	ISSN = {0025-5874,1432-1823},
	MRCLASS = {35Q30 (35B65 35K55 46E35 76B03 76D03)},
	MRNUMBER = {1980623},
	MRREVIEWER = {Bruno\ Scheurer},
	DOI = {10.1007/s002090100332},
	URL = {https://doi.org/10.1007/s002090100332},
}

\end{document}